\documentclass[12pt]{article}
\usepackage{amsmath,amsfonts,amsthm,amssymb,mathrsfs,latexsym,paralist, xypic}
\usepackage{soul}
\usepackage{tikz,hyperref,pgf}

\usetikzlibrary{arrows,automata}
\usepackage{slashed}

\tikzstyle{V}=[fill=black,circle,scale=0.2, outer sep = 4pt]

\usepackage{comment}

\newtheorem{thm}{Theorem}[section]
\newtheorem{prop}[thm]{Proposition}
\newtheorem{cor}[thm]{Corollary}
\newtheorem{lemma}[thm]{Lemma}
\theoremstyle{remark}
\newtheorem{rmk}[thm]{Remark}
\newtheorem{example}[thm]{Example}
\theoremstyle{definition}
\newtheorem{defn}[thm]{Definition}

\DeclareMathOperator{\diam}{diam}
\DeclareMathOperator{\Span}{span}

\newcommand{\bi}{\begin{itemize}}
\newcommand{\ei}{\end{itemize}}
\newcommand{\be}{\begin{enumerate}}
\newcommand{\ee}{\end{enumerate}}

\newcommand{\C}{\mathbb{C}}

\newcommand{\R}{\mathbb{R}}
\newcommand{\N}{\mathbb{N}}
\newcommand{\Z}{\mathbb{Z}}

\newcommand{\BB}{\mathcal{B}}
\newcommand{\Tr}{\operatorname{Tr}}
\newcommand{\Dom}{\operatorname{Dom}}

\newcommand{\Card}{\operatorname{Card}}

\providecommand{\keywords}[1]{{\textit{Key words and phrases:}} #1}
\providecommand{\classification}[1]{{\textit{2010 Mathematics Subject Classification:}} #1}

\usepackage[margin=1.1in]{geometry}

\AtEndDocument{\bigskip{\small%
  \textsc{Carla Farsi, Elizabeth Gillaspy, Sooran Kang, Judith Packer : Department of Mathematics, University of Colorado at Boulder, Boulder, Colorado, 80309-0395} \par
  \textit{E-mail address}: \texttt{farsi@euclid.colorado.edu, Elizabeth.Gillaspy@colorado.edu, \\ sooran@colorado.edu, packer@euclid.colorado.edu} \par
  \addvspace{\medskipamount}
  \textsc{Antoine Julien : Department of Mathematical Sciences, Norwegian University of Science and Technology 7491 Trondheim, Norway.} \par
  \textit{E-mail address}, \texttt{antoine.julien@math.ntnu.no}
}}

\begin{document}

\title{Wavelets and spectral triples for fractal representations of Cuntz algebras}
\author{C. Farsi, E. Gillaspy, A. Julien, S. Kang,  and J. Packer}
\date{\today}

\maketitle

\begin{abstract} In this article we provide an identification between the wavelet decompositions of certain fractal representations of $C^*$ algebras of directed graphs of M.~Marcolli and A.~Paolucci~\cite{MP}, and  the eigenspaces of Laplacians associated to spectral triples constructed from Cantor fractal sets that are the infinite path spaces of Bratteli diagrams associated to the representations, with a particular emphasis on wavelets for representations of Cuntz $C^*$-algebras  $\mathcal{O}_D$.   In particular, in this setting we use results of J.~Pearson and J.~Bellissard \cite{pearson-bellissard-ultrametric}, and A.~Julien and J.~Savinien~\cite{julien-savinien-transversal}, to construct first the spectral triple and then the Laplace--Beltrami operator on the associated Cantor set.  We then prove that in certain cases, the orthogonal wavelet decomposition and the decomposition via orthogonal eigenspaces  match up precisely.  We give several explicit examples, including an example related to a Sierpinski fractal, and  compute in detail all the eigenvalues and corresponding eigenspaces of the Laplace--Beltrami operators for the equal weight case for representations of ${\mathcal O}_D$, and in the uneven weight case for certain representations of ${\mathcal O}_2,$ and show how the eigenspaces and wavelet subspaces at different levels first constructed in~\cite{FGKP-survey} are related.
\end{abstract}

\classification{46L05.}

\keywords{Weighted Bratteli diagrams; Ultrametric Cantor set; Spectral triples; Laplace Beltrami operators; Wavelets.}

\tableofcontents

\section{Introduction}

In the 2011 paper~\cite{MP}, M.~Marcolli and A.~Paolucci, motivated by work of A.~Jonnson \cite{jonsson} and R.~Strichartz~\cite{Strichartz}, studied representations of Cuntz--Krieger $C^{\ast}$-algebras on Hilbert spaces associated to certain fractals, and constructed what they termed ``wavelets'' in these Hilbert spaces.  These wavelets were so called because they provided an orthogonal decomposition of the Hilbert space, and the partial isometries associated to the $C^*$-algebra in question gave  ``scaling and translation'' operators taking one orthogonal subspace to another. The results of Marcolli and Paolucci were generalized first to certain fractal representations of $C^{\ast}$-algebras associated to directed graphs and then to representations of higher-rank graph $C^{\ast}$-algebras $C^*(\Lambda)$ by some of the authors of this article in~\cite{FGKP} and~\cite{FGKP-survey}.  The $k$-graph $C^{\ast}$-algebras  $C^*(\Lambda)$ of Robertson and Steger~\cite{RS} are particular examples of these higher-rank graph algebras, and it was shown in~\cite{FGKP}  that for these Robertson--Steger $C^{\ast}$-algebras there is  a faithful representation of $C^*(\Lambda)$ on $L^2(X, \mu)$, where $X$ is a fractal space with Hausdorff measure $\mu.$ Moreover, this Hilbert space 
also admits a wavelet decomposition -- that is, an orthogonal decomposition such that the representation of $C^*(\Lambda)$ is generated by ``scaling and translation'' operators that move between the orthogonal subspaces.
As in Marcolli and Paolucci's original construction, the wavelets in  \cite{FGKP} and \cite{FGKP-survey} had a characteristic structure, in that they were chosen to be orthogonal to a specific type of function in the path space that could be easily recognized.

Earlier, the theory of spectral triples and Fredholm modules of A. Connes had generated great interest \cite{connes}, and such objects had been constructed for dense subalgebras of several different classes of $C^{\ast}$-algebras, including the construction of spectral triples by E.~Christensen and C.~Ivan on the $C^{\ast}$-algebras of Cantor sets~\cite{Chr-Iva}, which in turn motivated the work of J.~Pearson and J.~Bellissard, who constructed spectral triples and related Laplacians on ultrametric Cantor sets \cite{pearson-bellissard-ultrametric}. Expanding on the work of Pearson and Bellisard, A.~Julien and J.~Savinien studied similarly constructed Laplacians on fractal sets constructed from substitution tilings \cite{julien-savinien-transversal}.  In both the papers of Pearson and Bellissard and of Julien and Savinien, after the Laplacian operators were described, spanning sets of functions for the eigenspaces of the Laplacian were explicitly described in terms of differences of characteristic functions.  

It became apparent to the authors of the current paper that certain components of the  wavelet system as described in \cite{FGKP} and the explicit eigenfunctions given by Julien and Savinien in \cite{julien-savinien-transversal} seemed related, and one of the aims of this paper is to analyze this similarity in the case of $C^{\ast}$-algebras of directed graphs as represented on their infinite path spaces.  Indeed, we will show that under appropriate  hypotheses,   each orthogonal subspace described in the wavelet decomposition of~\cite{FGKP}  can be expressed as a union of certain of the eigenspaces of the Laplace--Beltrami operator from~\cite{julien-savinien-transversal}. 
{We suspect that the hypotheses required for this result can be substantially weakened from their statement in Theorem~\ref{MPwaveletsthm} below, and plan to explore this question in future work~\cite{FGKJP-spectral-triples}.}

More broadly, the goal of this  paper is to elucidate the connections between graph $C^*$-algebras, wavelets on fractals, and spectral triples.  We focus here on the case of one particular directed graph, namely the graph $\Lambda_D$ which has $D$ vertices and, for each pair ($v,w)$ of vertices, a unique edge $e$ with source $w$ and range $v$.  Again, many of the results presented here will hold in greater generality; see the forthcoming paper \cite{FGKJP-spectral-triples} for details.

In this paper we introduce the graph $C^*$-algebra (also known as a Cuntz algebra) associated to $\Lambda_D$; discuss the associated representations on fractal spaces as in \cite{MP, FGKP}; and present  the associated  spectral triples and 
Laplace--Beltrami operators associated to (fractal) ultrametric  Cantor sets 
as adapted from recent work  by Julien-Savinien, Pearson and Bellissard, Christensen et al., see e.g.~\cite{julien-savinien-transversal, pearson-bellissard-ultrametric, Chr-Ivan-Schr}. In particular we show in Theorem~\ref{MPwaveletsthm} that 
when one constructs the Laplace--Beltrami operator of \cite{julien-savinien-transversal} associated to the  infinite path space of $\Lambda_D$  (which is  an ultrametric Cantor set), the wavelets in \cite{MP} are exactly the eigenfunctions of the Laplacian.  We then   compute in detail all the eigenfunctions and eigenvalues of the 
 Laplace--Beltrami operator associated to a representation 
 of the Cuntz algebra  $\mathcal{O}_D$ on a Sierpinski type fractal set (see \cite{MP} Section 2.6 and Section \ref{sec:sierp-rep} below) for the definition of this representation). For several different choices of a measure on the infinite path space of $\Lambda_D$, we also compute all the  
the eigenfunctions and eigenvalues of the associated Laplace--Beltrami operator; in the case when $D=2$ and this measure arises from assigning the two vertices of $\Lambda_D$ the weights $r$ and $1-r$ for some $r \in [0,1]$, we 
compare these results to wavelets associated to certain representations of ${\mathcal O}_2$ analyzed in Section 3 of  \cite{FGKP-survey}.

In further work \cite{FGKJP-spectral-triples} we will generalize these constructions to more general directed graphs and to higher-rank graphs, and also explain how to generalize certain other spectral triples associated to directed graphs, such as those described in  \cite{CPR}, \cite{Chr-Iva}, \cite{goffme},  \cite{Whittaker}, and \cite{julien-putnam}, to higher-rank graphs. 

The structure of the paper is as follows.  In Section 2, we review the definition of directed graphs, with an emphasis on finite graphs and the construction of both the infinite path space and Bratteli diagrams associated to finite directed graphs, the first as described in \cite{MP} among other places, and the second as described in~\cite{IR}.  When the incidence matrices for our graphs are $\{0,1\}$ matrices, the infinite path space can defined in terms of both edges and vertices, and we describe this correspondence, together with the identification of the infinite path space  $\Lambda^{\infty}$ with the associated infinite path space of the Bratteli diagram $\partial \BB_\Lambda$ for a finite directed graph $\Lambda.$ In so doing, we note that these spaces are Cantor sets.  We also review the semibranching function systems of K. Kawamura \cite{kawamura} and Marcolli and Paolucci \cite{MP} in this section, with an emphasis on those systems giving rise to representations of the Cuntz algebras ${\mathcal O}_D.$  In Section 3, we review representations of ${\mathcal O}_D$ on the $L^2$-spaces of Sierpinski fractals first constructed by Marcolli and Paolucci in \cite{MP}, and show that these representations are equivalent to the standard positive monic representations of ${\mathcal O}_D$ defined by D. Dutkay and P. Jorgensen in \cite{dutkay-jorgensen-monic}.  In Section 4, we review the construction of spectral triples associated to weighted Bratteli diagrams, described by Pearson and Bellissard in \cite{pearson-bellissard-ultrametric} and Julien and Savinien in \cite{julien-savinien-transversal}, and provide explicit details of  their construction for a variety of weights on the Bratteli diagram $\partial \BB_D$ associated to the graph $\Lambda_D$.
We describe in Theorem \ref{thm-Dixmier-trace-cuntz-algebra-O_D} the conditions under which the measure on $\partial \BB_{D}$ agrees with the measure  introduced by Marcolli and Paolucci, which we describe in Section 2.   We also introduce the Laplace--Beltrami operator of Pearson and Bellissard \cite{pearson-bellissard-ultrametric}  in this setting and review the specific formulas for its eigenvalues and associated eigenspaces. In Section 5 we review the construction of Marcolli and Paolucci's wavelets associated to representations of Cuntz--Krieger $C^*$-algebras on the $L^2$-spaces of certain fractal spaces, with the notation for these subspaces provided in earlier papers~\cite{FGKP, FGKP-survey} with an emphasis on representations of the Cuntz $C^*$-algebra $\mathcal{O}_D$, and prove our main theorem (Theorem  \ref{MPwaveletsthm}), which is that in all cases that we consider, the wavelet subspaces for Marcolli and Paolucci's representations can be identified with the eigenspaces of the Laplace--Beltrami operator associated to the related Bratteli diagram. In Section 6, we examine certain representations of ${\mathcal O}_D$ where the weights involved are unevenly distributed among the vertices of $\Lambda_D$, and specializing to the study of uneven weights associated to representations of ${\mathcal O}_2,$ we compute explicitly the associated eigenvalues and eigenspaces for the Laplace--Beltrami operatore in this case, and provide the correspondence between these eigenspaces and certain wavelet spaces for monic representations of ${\mathcal O}_2$ first computed in~\cite{FGKP-survey}. 

 This work was partially supported by a grant from the Simons Foundation (\#316981 to Judith Packer).

\section{Cantor sets associated to directed graphs}\label{sec:Cantor_dir}
We begin with a word about conventions.  Throughout this paper, $\N$ consists of the positive integers, $\N = \{1, 2, 3, \ldots \}$; we use $\N_0$ to denote $\{0, 1, 2, 3 \ldots \}$.  The symbol $\Z_N$ indicates the set $\{0, \ldots, N-1\}$.  

The Bratteli diagrams we discuss below do not have a root vertex; indeed, we think of the edges in a Bratteli diagram as pointing towards the zeroth level of the diagram.  See Remark \ref{rmk-root} for more details.

\subsection{Directed graphs and Bratteli diagrams}
\begin{defn}
\label{def:directed-graph}
A  \emph{directed graph} $\Lambda$ consists of a set of vertices $\Lambda^0$ and  a set of edges $\Lambda^1$ 
and  range and source maps $r,s:\Lambda^1\to \Lambda^0$. We say that $\Lambda$ is \emph{finite} if 
\[
  \Lambda^n= 
    \{ e_1 e_2 \ldots e_n: e_i \in \Lambda^1, r(e_i) = s(e_{i-1}) \ \forall \; i\}
\]
 is finite for all $n\in\N$. If $\gamma = e_1 \cdots e_n$, we define $r(\gamma)= r(e_1)$ and $s(\gamma) = s(e_n)$, and we write $|\gamma| = n$. By convention, a path of length $0$ consists of a single vertex (no edge): if $|\gamma| = 0$ then $\gamma = (v)$ for some vertex $v$.
 
 We say that $\Lambda$ has \emph{no sources} if $v\Lambda^n =\{\gamma\in\Lambda^n: r(\gamma)=v\}\ne \emptyset$  for all $v\in\Lambda^0$ and all $n \in \N$. We say that $\Lambda$ is \emph{strongly connected} if 
 \[
 v\Lambda w = \bigcup_{n\in \N} \{\gamma \in v \Lambda^n: s(\gamma) = w \}\ne \emptyset\]
  for all $v,w\in\Lambda^0$.
 In a slight abuse of notation, if  $\Lambda^n$ denotes the set of finite paths of length $n$, we denote by $\Lambda = \cup_{n\in \N_0} \Lambda^n$ the set of all finite paths,  and by $\Lambda^\infty$ the set of infinite paths of a finite directed graph $\Lambda$:
 \[
   \Lambda^\infty = \biggl\{ (e_i)_{i\in\N} \in \prod_{i=1}^\infty \Lambda^1: s(e_i) = r(e_{i+1}) \ \forall \ i \in \N  \biggr\}.
 \]
 For $\gamma \in \Lambda$, we write $[\gamma] \subseteq \Lambda^\infty$ for the set of infinite paths with initial segment~$\gamma$:
 \begin{equation}\label{eq:cylin}
   [e_1 \ldots e_n]  = \bigl\{  (f_i)_i \in \Lambda^\infty: f_i = e_i \ \forall \ 1 \leq i \leq n \bigr\}.
 \end{equation}
 We say that  a path $\gamma = e_1 \ldots e_n$ has length $n$ and write $|\gamma| = n$. If $\gamma = (v)$ is a path of length 0, then $[\gamma] = [v] = \{ (f_i)_i \in \Lambda^\infty : r(f_1) = v \}$.
 
  Given a finite directed graph $\Lambda$, the \emph{vertex matrix} $A$ of $\Lambda$ is an $\Lambda^0  \times \Lambda^0$ matrix with entry $A({v,w}) = | v\Lambda^1 w|$ counting the number of edges with range $v$ and source $w$ in $\Lambda$. 
\end{defn} 

\begin{rmk}
\label{rmk:top-on-inf-path}
 As shown in \cite{kprr} Corollary 2.2, if $\Lambda$ is finite and source-free, the cylinder sets $\{ [\gamma]: \gamma \in \Lambda\}$ form a compact open basis for a locally compact, totally disconnected, Hausdorff topology on $\Lambda^\infty$.\footnote{Note that if $\Lambda$ is finite, it is also row-finite, according to the definition given in Section 2 of \cite{kprr}.}  If $\Lambda$ is finite, $\Lambda^\infty$ is also compact. 
\end{rmk}

According to \cite{aHLRS3} Proposition 8.1, a strongly connected finite directed graph $\Lambda$ has a 
distinguished Borel measure $M$ on the infinite path space $\Lambda^\infty$ which is given in terms of the spectral radius $\rho(A)$ of the vertex matrix $A$;
\begin{equation}\label{eq:measure}
M([\gamma])=\rho(A)^{{-\vert\gamma\vert}}P_{s(\gamma)},
\end{equation}
where 
$(P_v)_{v\in\Lambda^0}$ is  the unimodular Perron--Frobenius eigenvector of the vertex matrix $A$. (See section~2 of \cite{FGKP} for details).

\begin{defn}
\label{def:bratteli-diagram}
Let $\Lambda$ be a finite directed graph with no sources. 
The \emph{Bratteli diagram} associated to $\Lambda$ is an infinite directed graph $\BB_\Lambda$, with the set of vertices $V=\sqcup_{n\ge 0} V_n$ and the set of edges $E=\sqcup_{n\ge 1}E_n$ such that
\begin{itemize}
\item[(a)] For each $n\in \N_0$, $V_n\cong\Lambda^0$ and $E_{n+1}\cong\Lambda^1$.
\item[(b)] There are a range map and a source map $r,s:E\to V$ such that $r(E_n)\subseteq V_{n-1}$ and $s(E_n) \subseteq V_n$ for all $n\in \N$. 
\end{itemize}
 A \emph{path} $\gamma$ of length $n\in\N$ in $\BB_\Lambda$ is an element 
\[
e_1e_2\dots e_n =(e_1,e_2,\dots, e_n)\in \prod_{i=1}^{n}E_n
\]
which satisfies $|e_i|=1$ $\forall i $, and  $s(e_i)=r(e_{i+1})$ for all $1\le i\le n-1$. We denote by $F\BB_\Lambda$ the set of all finite paths in the Bratteli diagram $\BB_\Lambda$, and by $F^n\BB_\Lambda$ the set of all finite paths in the Bratteli diagram $\BB_\Lambda$ of length $n$.

We denote by $\partial \BB_\Lambda$ the set of infinite paths in the Bratteli diagram $\BB_\Lambda$; 
\[
\partial \BB_\Lambda = \{ e_1e_2\dots =(e_1,e_2,\dots)\in \prod_{n=1}^{\infty}E_n : |e_i|=1,\  s(e_i) = r(e_{i+1}) \ \forall \ i \in \N\}.
\]
Given a (finite or infinite) path   $\gamma = e_1 e_2 \ldots$  in $\BB_\Lambda$ and $m \in \N$, we write 
\[\gamma[0,m]= e_1 e_2 \cdots e_m.\]
If $m = 0$ we write $\gamma[0,0] = r(\gamma).$
\end{defn}

\begin{rmk}
{Any finite path $\gamma$ of a length $n$ in a directed graph (or a Bratteli diagram) is given by a string of $n$ edges $e_1e_2\dots e_n$, which can be written uniquely as a string of vertices $v_0v_1\dots v_n$ such that $r(e_i)=v_{i-1}$ and $s(e_i)=v_i$ for $1\le i\le n$. Conversely,  if the vertex matrix $A$ has all entries either $0$ or 1 (as will be the case in all of our examples), a given string of vertices $v_0 v_1 \ldots v_n$ with $v_i \in V_n$ for all $n \in \N_0$ corresponds to at most one string of edges, and hence at most one finite path $\gamma$. Thus even though our formal definition of a path is given as a string of edges, sometimes we use the notation of a string of vertices for a path.}
\end{rmk}

\begin{rmk}
\label{rmk-root}
{Note that our description of a Bratteli diagram is different from the 
one in \cite{julien-savinien-transversal} and \cite{bezuglyi-jorgensen}. First, the edges in $E_n$ in \cite{julien-savinien-transversal} and in \cite{bezuglyi-jorgensen} have source in $V_n$ and range in $V_{n+1}$; in other words, they point in the opposite direction from our edges.  More substantially, though, in \cite{julien-savinien-transversal} and \cite{bezuglyi-jorgensen} every finite (or infinite) path in a Bratteli diagram  starts from a vertex called a root vertex, $\circ$, and any finite path that ends in $V_n$ is given by $\epsilon_{r(e_1)} e_1 e_2\dots e_n$, where for each vertex $v \in V_0$,  there is a unique edge $\epsilon_v$ connecting $\circ$ and $v$. This implies that a finite path that ends in $V_n$ consists of $n+1$ edges in their Bratteli diagram. However, our description of a Bratteli diagram in Definition \ref{def:bratteli-diagram} does not include a root vertex, and a finite path 
that ends in $V_n$ consists of $n$ edges. 
Thus, when we discuss Theorem~4.3 of \cite{julien-savinien-transversal} in Sections \ref{subsec-Laplace-Beltrami Operator-O-D} and \ref{sec:spect-triples-O-2} below, we will need to  introduce a single path, the ``empty path'' of length -1, which we will denote by $\gamma[0,-1]$ for any and all  paths $\gamma \in F\BB_\Lambda$.  The cylinder set of this path is $[\circ] = \partial \BB_\Lambda$ when we translate Theorem~4.3 of \cite{julien-savinien-transversal} to our setting. }
\end{rmk}

\begin{rmk}
\label{rmk:ident-graph-bratteli}
As is suggested by the notation, a finite directed graph and its associated Bratteli diagram encode the same information in their sets of finite and infinite paths.  We wish to emphasize this correspondence in this paper, to illuminate the way tools from a variety of disciplines combine to give us information about wavelets on fractals.
\end{rmk}

\begin{rmk}
If $\Lambda$ is a strongly connected finite directed graph, then $\Lambda$ has no sources by Lemma 2.1 of \cite{aHLRS3}. Hence every vertex of the associated Bratteli diagram $\BB_\Lambda$ also receives an edge.
\end{rmk}

\begin{example}\label{ex1}
Consider a directed graph $\Lambda$ with two vertices $v,w$ and four edges $f_1,f_2, f_3$ and $f_4$ given as follows:
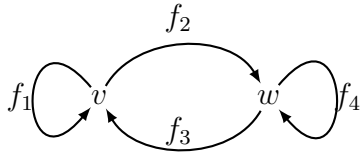
\[
\begin{tikzpicture}[scale=1.5]
 \node[inner sep=0.5pt, circle] (v) at (0,0) {$v$};
    \node[inner sep=0.5pt, circle] (w) at (1.5,0) {$w$};
    \draw[-latex, thick] (w) edge [out=50, in=-50, loop, min distance=30, looseness=2.5] (w);
    \draw[-latex, thick] (v) edge [out=130, in=230, loop, min distance=30, looseness=2.5] (v);
\draw[-latex, thick] (w) edge [out=240, in=-60] (v) ;
\draw[-latex, thick] (v) edge [out=60, in=120] (w);
\node at (-0.7, 0) {\color{black} $f_1$}; 
\node at (0.7, 0.7) {\color{black} $f_2$};
\node at (0.7, -0.3) {\color{black} $f_3$};
\node at (2.2, 0) {\color{black} $f_4$};
\end{tikzpicture}
\]
Note that $\Lambda$ is finite and strongly connected, and (consequently) has no sources. The vertex matrix $A$ is given by
\[
A=\begin{pmatrix} 1 & 1\\ 1& 1\end{pmatrix},
\]
and the associated Bratteli diagram $\BB_\Lambda$ is 

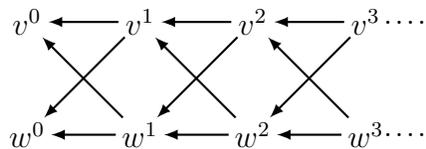
\[
\begin{tikzpicture}[scale=1.5]
\node[inner sep=0.5pt, circle] (v0) at (0,1) {$v^0$};
\node[inner sep=0.5pt, circle] (w0) at (0,0) {$w^0$};
\node[inner sep=0.5pt, circle] (v1) at (1,1) {$v^1$};
\node[inner sep=0.5pt, circle] (w1) at (1,0) {$w^1$};
\node[inner sep=0.5pt, circle] (v2) at (2,1) {$v^2$};
\node[inner sep=0.5pt, circle] (w2) at (2,0) {$w^2$};
\node[inner sep=0.5pt, circle] (v3) at (3,1) {$v^3$};
\node[inner sep=0.5pt, circle] (w3) at (3,0) {$w^3$};
\draw[-latex, thick] (v1) edge (v0);    
\draw[-latex, thick] (v1) edge (w0);
\draw[-latex, thick] (w1) edge (v0);    
\draw[-latex, thick] (w1) edge (w0);
\draw[-latex, thick] (v2) edge (v1);    
\draw[-latex, thick] (v2) edge (w1);
\draw[-latex, thick] (w2) edge (v1);    
\draw[-latex, thick] (w2) edge (w1);
\draw[-latex, thick] (v3) edge (v2);    
\draw[-latex, thick] (v3) edge (w2);
\draw[-latex, thick] (w3) edge (v2);    
\draw[-latex, thick] (w3) edge (w2);
\node at (3.2, 1) {$.$}; \node at (3.3,1) {$.$}; \node at (3.4, 1) {$.$}; \node at (3.5,1) {$.$};
\node at (3.2, 0) {$.$}; \node at (3.3,0) {$.$}; \node at (3.4, 0) {$.$}; \node at (3.5,0) {$.$};
\end{tikzpicture}
\]
\end{example}

\begin{prop}
 Let $\Lambda$ be a finite directed graph. If every vertex $v$ in 
 the directed graph $\Lambda$ receives two distinct infinite paths, then $\Lambda^\infty$ (equivalently, $\partial \BB_\Lambda$) has no isolated points and hence it is a Cantor set. 
\end{prop}
 
\begin{proof}
 Recall that a Cantor set is a totally disconnected, compact, perfect topological space.    Moreover, $\Lambda^\infty$ is always compact Hausdorff and totally disconnected by Corollary 2.2 of \cite{kprr}, so it will suffice to show that $\Lambda^\infty$ has no isolated points.

 Suppose $\Lambda^\infty$ has an isolated point $ (e_i)_{i\in\N}$.  Since the cylinder sets form a basis for the topology on $\Lambda^\infty$, this implies that there exists $n \in \N$ such that $[e_1 \cdots e_n]$ only contains $(e_i)_{i\in\N}$.  In other words, for each $m \geq n$, there is only one infinite path with range  $ s(e_m)$, contradicting the hypothesis of the proposition. 
\end{proof}

\begin{cor}
\label{cor:suff-cond-cantor}
 If $\Lambda$ is a finite directed graph with $\{0,1\}$ vertex matrix $A$ and  every row sum of $A$ is at least 2, then $\Lambda^\infty$ 
 (equivalently, $\partial \BB_\Lambda$) is a Cantor set. 
\end{cor}

\begin{proof}
 Note that the sum of the $v$th row  of $A$ represents the number of edges in $\Lambda$ with range $v$.  If every vertex receives at least two edges, then any cylinder set $[\gamma]$ will contain infinitely many elements, so $\Lambda^\infty$ has no isolated points.  
\end{proof}
Corollary \ref{cor:suff-cond-cantor} tells us that the infinite path space of Example~\ref{ex1} is a Cantor set.

\subsection{Cuntz algebras and representations on fractal spaces}
\begin{defn}[{\cite[Definition~2.1]{dutkay-jorgensen-monic}}] Fix an integer $D>1.$ The {\em Cuntz algebra ${\mathcal O}_D$} is the universal $C^{\ast}$-algebra generated by  isometries 
$\{T_i\}_{i=0}^{D-1}$ satisfying {the Cuntz relations}
\begin{equation}
\label{CK1}
T_j^{\ast}T_i\;=\;\delta_{ij}\text{I}, 
\end{equation}
and
\begin{equation}
\label{CK3}
\sum_{i=0}^{D-1}T_iT_i^{\ast}\;=\;\text{I}.
\end{equation}
\end{defn}
{The above definition of $\mathcal{O}_D$ is equivalent to the definition of $\mathcal{O}_{A_D}$ in the beginning of section 2 of \cite{MP} associated to the matrix $A_D$ that is a $D\times D$ matrix with 1 in every entry:
\begin{equation}\label{matrixAD}
   A_{D} = \begin{pmatrix} 1 & 1 & 1 & ...&1\\
                            1 & 1 & 1 & ...&1\\
                   \vdots & \vdots & \vdots & \vdots & \vdots \\
                            1 & 1 & 1 & ...&1 \\
                            1 & 1 & 1 & ...&1\end{pmatrix}.
\end{equation}

As had been done previously by K. Kawamura \cite{kawamura},  Marcolli and Paolucci constructed representations of $\mathcal{O}_D$ (and more generally, the Cuntz--Krieger algebras ${\mathcal O}_A$ associated to a matrix $A$) by employing the method of ``semibranching function systems.''  
We note for completeness that the semibranching function systems of Kawamura~\cite{kawamura} were for the most part defined on finite Euclidean spaces, e.g.\@ the unit interval $[0,1],$ whereas the semibranching function systems used by Marcolli and Paolucci~\cite{MP} were mainly defined on Cantor sets.

\begin{defn}[{cf.~\cite{kawamura}, \cite[Definition 2.1]{MP}, \cite[Definition 2.16]{bezuglyi-jorgensen} 
}] 
\label{semibranchingdef}
Let $(X,\mu)$ be a measure space, {fix an integer $D>1$ and let} 
$\{\sigma_i:X\to X\}_{i\in\mathbb Z_D}$ be a collection of $\mu$-measurable maps. 
The family of maps  $\{\sigma_i\}_{i\in\mathbb Z_D}$ is called a \emph{semibranching function system}
on $(X,\mu)$ with coding map $\sigma:X\to X$ 
if the following conditions hold:
\begin{enumerate}
  \item For $i\in \Z_D$, set $R_{[i]}=\sigma_i(X)$. Then we have 
  \[
  \mu(X\backslash \cup_{i\in\mathbb Z_D}R_{[i]})=0\;\;\;\text{and}\;\;\; \mu(R_{[i]}\cap R_{[j]})=0\;\;\text{for}\;\;i\not=j.
  \]
  \item  For $i\in \Z_D$, we have $\mu\circ \sigma_i \ll \mu$ and {the Radon--Nikodym derivative} satisfies
 \begin{equation}
 \label{semibranching}
 \frac{d(\mu\circ \sigma_i)}{d\mu}\;>\;0,\;\mu\text{-a.e.}
 \end{equation}
  \item  For all $i\in \Z_D$, we have $$\sigma\circ \sigma_i(x)\;=\;x,  \ \text{$\mu$-a.e}.$$
\end{enumerate}

\end{defn}
 
Kawamura and then Marcolli and Paolucci observed the following relationship between semibranching function systems 
and representations of ${\mathcal O}_D:$

\begin{prop}[{cf.~\cite[Proposition 2.4]{MP}, \cite[Theorem~2.22]{dutkay-jorgensen-monic}}]
\label{MPrepprop}
Let $(X,\mu)$ be a measure space, and let $\{\sigma_i:X\to X\}_{i\in\mathbb Z_D}$
be a semibranching function system 
on $(X,\mu)$ with coding map $\sigma:X\to X$. 
For each $i\in \mathbb Z_D$ define $S_i:L^2(X,\mu)\to L^2(X,\mu)$ by
$$S_i(\xi)(x)\;=\;\chi_{R_{[i]}}(x)\Big(\frac{d\mu\circ \sigma_i}{d\mu}(\sigma(x))\Big)^{-\frac{1}{2}}\xi(\sigma(x))\;\;\text{for $\xi\in L^2(X,\mu)$ and $x\in X$}.$$
Then the family $\{S_i\}_{i\in \mathbb Z_D}$  satisfies the 
 Cuntz relations Equations (\ref{CK1}) and (\ref{CK3}), and therefore generates a representation of the Cuntz algebra ${\mathcal O}_D.$
\end{prop}

\begin{example}
\label{ex:inf-path-repn}
Let $\Lambda_D$ be the directed graph associated to the vertex matrix $A_D$. 
We can define a semibranching function system $\{(\sigma_i)_{i\in \Z_D}, \sigma\}$ on the Cantor set $(\Lambda_D^\infty, M)$ by thinking of  elements of $\Lambda_D^\infty$ as sequences of vertices $(v_i)_{i \in \N_0}$ with $v_j \in \Z_D \ \forall \, j$.  With this convention, we set 
\[ \sigma_i (v_0 v_1 v_2 \ldots ) = (i v_0 v_1 v_2 \ldots) \text{ and } \sigma(v_0 v_1 \ldots) = (v_1 v_2 \ldots).\]
Then the Radon--Nikodym derivative $\frac{d(M \circ \sigma_i)}{dM}$ is given by 
\[ \frac{d(M \circ \sigma_i)}{dM} = \frac{1}{D}\]
since the cylinder set $R_{[i]}$ has measure $\frac{1}{D}$ for all $i$,
and the associated operators $S_i$ are given by 
\[ S_i (\xi)( v_0 v_1 v_2 \ldots) = \left\{ \begin{array}{cl} \sqrt{D} \xi( v_1 v_2 \ldots) & \text{ if } v_0 = i \\
0 & \text{ else.} 
\end{array} \right. \]
This representation of $\mathcal{O}_D$ is faithful by Theorem 3.6 of \cite{FGKP}, since every cycle in $\Lambda_D$ has an entrance.
\end{example}

\begin{example}[{cf.~\cite[Proposition 2.6]{MP}}]
\label{ex:mp-fractal}
  Take {an integer} $D>1$, and let $K_D = \prod_{j=1}^{\infty}[\mathbb Z_D]_j$, {which is called the \emph{Cantor group} on $D$ letters in Definition~2.3 of~\cite{dutkay-jorgensen-monic}. As described in Section~2 of \cite{FGKP-survey}, $K_D$ has a Cantor set topology which is generated by cylinder sets
\[
[n]=\{(i_j)_{j=1}^\infty\in K_D : i_1=n\}.  
\] 
According to Section~3 of \cite{dutkay-jorgensen-monic}, there is a measure $\nu_D$ on $K_D$ given by
  \[
  \nu_D([n_1n_2\dots n_m])=\prod_{j=1}^m \frac{1}{D}=\frac{1}{D^m}.
  \]
 Note that $\nu_D$ is a Borel measure on $K_D$ with respect to the cylinder-set Cantor topology. }
 
For each $j \in \Z_D$,  define $\sigma_j$ {on $K_D$} by
\[
  \sigma_j\left( (i_1i_2\cdots i_k\cdots )\right )\;=\; (ji_1i_2\cdots i_k\cdots ).
\]
Then
\[
  R_{[j]} = {\sigma_j(K_D)}=\{(j i_1 i_2\cdots i_k\cdots ) \; : (i_1 i_2 \cdots i_k \cdots ) \in K_D\}=[j],
\]
and,
denoting by $\sigma$ the one-sided shift on $K_D, \ \sigma\left( (i_1i_2\cdots i_k\cdots )\right )= (i_2i_3\cdots i_{k+1}\cdots ),$
we have that $\sigma\circ \sigma_j(x)=x$  for all $x\in K_D$ and $j\in \mathbb Z_D.$
Marcolli and Paolucci show in Section 2.1 of \cite{MP} that this data  gives a semibranching function system.   Moreover, since the measure of each set $R_{[i]}$ is $\frac{1}{D}$, the Radon--Nikodym derivative $\frac{d(\nu_D\circ \sigma_i)}{d\nu_D}$ satisfies
\[
 \frac{d(\nu_D\circ \sigma_i)}{d\nu_D}\;=\;\frac{1}{D}.
\]
{Thus, Proposition~\ref{MPrepprop} implies that there is a family of operators $\{S_i\}_{i\in \Z_D} \subseteq B(L^2(K_D, \nu_D))$ that generates a representation of the Cuntz algebra $\mathcal{O}_D$.}  

Moreover, this representation is faithful by Theorem 3.6 of \cite{FGKP}. To see this, let $\Lambda_D$ denote the directed graph with vertex matrix $A_D$, and note that labeling the vertices of $\Lambda_D$ by $\{0, 1, \ldots, D-1\}$ allows us to identify an infinite path $(e_i)_{i\in \N} \in \partial \BB_D$ with the sequence $(r(e_i))_{i\in \N} \in K_D$.  Moreover, in this case the Perron--Frobenius eigenvector associated to $A_D$ is 
\[P = \biggl( \frac{1}{D}, \frac{1}{D}, \ldots, \frac{1}{D} \biggr) ,\]
and consequently 
 \[ M([e_1 \ldots e_n]) = \frac{1}{D^{n+1}} = \nu_D\bigl( [r(e_1) r(e_2) \cdots r(e_n) s(e_n)] \bigr).\]
Since the cylinder sets generate the topology on both $K_D$ and on $\partial \BB_D$, this identification is measure-preserving.  Thus, the representation $\{S_i\}_{i \in \Z_D}$ of $\mathcal{O}_D$ on $L^2(K_D, \nu_D)$ is equivalent to the infinite path representation of Example \ref{ex:inf-path-repn}.   We can apply Theorem 3.6 of \cite{FGKP} to this latter representation to conclude that it is faithful, since every cycle in the graph $\Lambda_D$ associated to $A_D$ has an entry.  
In other words, 
\[C^*\bigl( \{S_i\}_{i\in \Z_D} \bigr) \cong \mathcal{O}_D.\]
\end{example}

\section{The action of ${\mathcal O}_D$ on $L^2({\mathbb S}_A,H)$}\label{subs:Sierp-fractal}
As mentioned in the Introduction, we wish to show that when we represent $\mathcal{O}_D$ on a 2-dimensional Sierpinski fractal $\mathbb{S}_A$, this representation of $\mathcal{O}_D$ also gives rise to wavelets.  We will then  compare these wavelets with the eigenfunctions of the Laplace--Beltrami operator $\Delta_s$ of \cite{julien-savinien-transversal} that is associated to {$A_D$,} the $D \times D$ matrix of all 1's (that is, the matrix associated to the Cuntz algebra $\mathcal{O}_D$). To compare these functions, we will establish a measure-preserving isomorphism between $\mathbb{S}_A$ and the infinite path space of the directed graph (equivalently, Bratteli diagram) associated to $\mathcal{O}_D$ in this section. (See Theorem~\ref{thm:measure_preserving} below).

\subsection{The Sierpinski fractal representation for $\mathcal{O}_D$}

\label{sec:sierp-rep}
Let $N$ and $D$ be positive integers with $N\geq 2,$ and let $A$ be a $N\times N \{0,1\}$-matrix with exactly $D$ entries consisting of the number $1$.  Suppose that the nonzero entries of $A$ are in positions  $\{(a_j,b_j)\}_{j=0}^{D-1},$ where $a_j, b_j\in\{0,1\cdots,N-1\}$ and in a lexicographic ordering we have $(a_0,b_0)<(a_1,b_1)<\cdots <(a_{D-1},b_{D-1}).$ Here we say $(a,b)<(a',b')$ if either $a<a'$ or if $a=a'$ and $b<b'.$

In Section~2.6 of \cite{MP}, Marcolli and Paolucci defined the Sierpinski fractal associated to $A,\;\mathbb{S}_A\;\subset [0,1]^2,$ as follows:
\[
  {\mathbb S}_A = \;  \biggl\{ (x,y)= \biggl( \sum_{i=1}^\infty \frac{x_i}{N^i}, \sum_{i=1}^\infty \frac{y_i}{N^i} \biggr) 
   : \; x_i, y_i \in \mathbb Z_N, \; A_{x_i,y_i}=1, \; \forall i\in \mathbb N  \biggl\}.
\]
For each $j \in \Z_D$, we define

\begin{equation}\label{eq:MPsemi_fts_Sierp}
 \tau_{j}(x,y) = \biggl( \frac{x}{N}+\frac{a_j}{N}, \frac{y}{N}+\frac{b_j}{N} \biggr) \quad \text{and}\quad
 \tau(x,y) = \biggl( N \Bigl( x-\frac{x_1}{N} \Bigr), N \Bigl( y-\frac{y_1}{N} \Bigr) \biggr).
\end{equation}

Lemma 2.23 of \cite{MP} tells us that the operators $\{\tau_j\}_{j \in \Z_D}$ form a semibranching function system with coding map $\tau$, and hence determine a representation of the Cuntz algebra $\mathcal{O}_D$  associated to $A_D$ given in \eqref{matrixAD}, on the Hilbert space $L^2(\mathbb{S}_A, H)$.  
Here $H$ is the Hausdorff measure on the fractal $\mathbb{S}_A$. 

According to the work of Hutchinson \cite{hutch}, we have

\[
 {\mathbb S}_A = \bigcup_{i=1}^D \tau_j({\mathbb S}_A).
\]

Moreover, the work of \cite{hutch} shows that the Hausdorff measure $H$ on $\mathbb{S}_A$ 
is the unique Borel probability measure on ${\mathbb S}_A$ satisfying the self-similarity equation
\begin{equation}\label{eq:hausd-meas-scaling}
  H = \sum_{i=0}^{D-1} \frac{1}{D}(\tau_j)_*(H).
\end{equation}
In other words, 
\[
 H(\tau_j({\mathbb S}_A)) = \frac{1}{D}H({\mathbb S}_A)) = \frac{1}{D}.
\]
It follows that, since 
\[
  \tau_j ({\mathbb S}_A)) 
    = \biggl\{ \biggl( \sum_{i=1}^\infty \frac{x_i}{N^i}, \sum_{i=1}^\infty \frac{{y_i}}{N^i} \biggr): (x_1,y_1)=(a_j,b_j) \biggr\},
\]
\[
 H\biggl( \biggl\{ \sum_{i=1}^\infty \frac{x_i}{N^i},\sum_{i=1}^\infty \frac{y_i}{N^i})\in {\mathbb S}_A: (x_1,y_1)=(a_j,b_j) \biggr\} \biggr) = \frac{1}{D}.
\]
By repeatedly applying the measure-similitude equation \eqref{eq:hausd-meas-scaling} we obtain 
\begin{multline}
\label{eq:sierp-meas-formula}
H\biggl( \biggl\{ \Bigl( \sum_{i=1}^\infty \frac{x_i}{N^i},\sum_{i=1}^\infty \frac{y_i}{N^i} \Bigr) \in {\mathbb S}_A: 
  \forall 1 \leq i \leq M, \ (x_i,y_i)=(a_{j_i},b_{j_i}) \biggr\} \biggr)\\
 = H(\tau_{j_1}\circ\tau_{j_2}\circ \cdots \circ \tau_{j_M}(\mathbb S)_A)=\Big(\frac{1}{D}\Big)^M.
\end{multline}

\subsection{The measure-preserving isomorphism}
In this section, we discuss in more detail the relationship between the representation of $\mathcal{O}_D$ on $L^2(\mathbb{S}_A, H)$ and the  infinite path
 representation of $\mathcal{O}_D$ on $L^2(\partial \BB_D, M)$  described in Example~\ref{ex:inf-path-repn}.

First, we note that the Hausdorff dimension of the Sierpinski fractal $\mathbb{S}_A$ introduced above is 
$$\frac{\ln D}{\ln N},$$
as established in Hutchinson's paper \cite{hutch}.\footnote{This formula is not in line with~\cite[Equation~(2.64)]{MP}, which gives $\ln D /( 2 \ln N)$ for the Hausdorff dimension. However, said equation appears to be a typo: the dimension should be~$2$ when $D = N^2$ (i.e.\@ when $\mathcal S_A$ is the unit square).}
In particular, in the classical case of the Sierpinski triangle corresponding to the $2\times 2$ matrix 
$A = \begin{pmatrix} 1 & 0 \\ 1 & 1 \end{pmatrix},$
the Hausdorff dimension of $\mathbb{S}_A$ is
$\frac{\ln 3}{\ln 2}.$

The main goal of this section is to prove the following:

\begin{thm}\label{thm:measure_preserving}
 Let $A$ be the $N\times N$ matrix with entries consisting of only $0$'s and $1$'s with $D$ incidences of $1$'s in the entries $$(a_0,b_0)<(a_1,b_1)<\cdots <(a_{D-1},b_{D-1}),$$ where $a_j,b_j\in \mathbb Z_N.$
 Consider the Sierpinski gasket fractal 
 \[
   {\mathbb S}_A\;=\; \biggl\{ \Bigl( \sum_{i\in \N}\frac{x_i}{N^i}, \sum_{i\in \N}\frac{y_i}{N^i} \Bigr) :\;A(x_i,y_i)=1,\;\forall i\in\mathbb N\biggr\}.
 \]
 Then there is a measure-theoretic isomorphism 
 \[
  \Upsilon = \Phi \circ \Theta: (\partial \BB_D, M)\to ({\mathbb S}_A,H),
 \]
 where $(\partial \BB_D,M)$ is the infinite path space of the Bratteli diagram associated to the $D \times D$ matrix with all ones, and $M$ is the  measure given by Equation \eqref{eq:measure}:
 \[
  M[\gamma] = D^{-|\gamma| -1}.
 \]
 Moreover, if $\{S_i\}_{i \in \Z_D}$ denotes the infinite path representation of $\mathcal{O}_D$ on $(\partial \BB_\Lambda, M)$, and $\{T_i\}_{i \in \Z_D}$ denotes the representation of $\mathcal{O}_D$ on $(\mathbb{S}_A, H)$ associated to the semibranching function system \eqref{eq:MPsemi_fts_Sierp}, then  for all $i \in\Z_D$,
 \[
  T_i = S_i\circ \Upsilon.
 \]
\end{thm}

\begin{proof}
Let $S_A$ denote the $D$-element symbol space of pairs from $\mathbb Z_N$ with $1$'s in the corresponding entry of $A:$ 
$$S_A=\{(a_0,b_0), (a_1,b_1), (a_2,b_2),\cdots, (a_{D-1},b_{D-1})\}\subset \mathbb Z_N\times \mathbb Z_N,$$
and let $X_A$ be the infinite product space $X_A=\;\prod_{i=1}^{\infty}S_A.$  Giving $S_A$ the discrete topology and $X_A$ the product topology, we see that $X_A$ is  a 
Cantor set, by the arguments of Section 2 of \cite{FGKP-survey}.  For every $i\in \mathbb N,$ let $\mu_{i,A}$ be the normalized counting measure on $S_A$; that is,  for $S\subset S_A,$
$$\mu_{i,A}(S)\;=\; \frac{\#(S)}{D},$$
and let $\mu_A$ denote the infinite product measure $\mu_A=\prod_{i=1}^{\infty}(\mu_{i,A}).$
Note if we let 
\begin{equation}\label{eq:Serpin_cylin}
 \bigl[(a_{j_1},b_{j_1})(a_{j_2},b_{j_2})\cdots (a_{j_M},b_{j_M})\bigr]
\end{equation}
 denote the cylinder set 
\[\begin{split}
&[(a_{j_1},b_{j_1})(a_{j_2},b_{j_2})\cdots (a_{j_M},b_{j_M})]\\
&\;\;=\{\left((x_i,y_i)\right)_{i=1}^{\infty}\in X_A: (x_i,y_i)=(a_{j_i},b_{j_i}) \ \forall \ 1\leq i\leq M\},
\end{split}\]
then 
$$\mu_A([(a_{j_1},b_{j_1})(a_{j_2},b_{j_2})\cdots (a_{j_M},b_{j_M})])=\frac{1}{D^M}.$$
Define now a map $\Phi: X_A\to {\mathbb S}_A$ by 
\[
 \Phi\bigl( \left((x_i,y_i)\right)_{i=1}^{\infty} \bigr) =  \Bigl( \sum_{i=1}^{\infty}\frac{x_i}{N^i}, \sum_{i=1}^{\infty}\frac{y_i}{N^i} \Bigr).
\]
The map $\Phi$ is continuous from the product topology on $X_A$ to the topology on ${\mathbb S}_A$ inherited from the Euclidean topology on $[0,1]\times [0,1].$  The map $\Phi$ is {\bf not} one-to-one, but if we let $E\subset X_A$ denote the set of points on which $\Phi$ is not injective, 
$\mu_A(E)=0$.
Indeed, let's examine the set of points of $X_A$ where $\Phi$ may not be one-to-one: non-injectivity can come from pairs of sequence of the forms $(x_i, y_i)_i$, $(x'_i, y'_i)_i$ where $x_i$ is eventually $N-1$ and $x'_i$ is eventually $0$, 
and similarly exchanging $x$ and $y$.  
Notice also that if $A$ has no ones either on the first or on the last row, there will be no such pairs for which $x_i$ is eventually $N-1$ and $x'_i$ is eventually $0$. Therefore, since $A$ has $D$ total entries equaling 1, if two such pairs $(x_i, y_i)_i$ and $(x'_i, y'_i)_i$ are going to have the same image under $\Phi$, there need to be at most $D-1$ ones on the first row, and the same on the last row.
Therefore, the measure of the set of pairs $(x_i, y_i)_i$ for which $x_i$ is eventually $N-1$ is smaller than $[(D-1)/D]^n$ for all $n$: it has zero measure. We reason similarly for the set of pairs $(x_i, y_i)$ for which $x_i$ is eventually $0$, for which $y_i$ is eventually $0$ and for which $y_i$ is eventually $N-1$. In conclusion, the set of points in $X_A$ on which $\Phi$ has a risk of not being one-to-one has measure zero.

We also note that since $\Phi$ is continuous, it is a Borel measurable map, and that for any Borel subset $B$ of $\mathbb{S}_A$, 
$$\mu_A\circ [\Phi]_*(B)\;=H(B).$$
This is the case because a length-$M$ cylinder set in ${\mathbb S}_A$ (that is, any cylinder  set \\
$\bigl[ (x_1, y_1), \ldots, (x_M, y_M) \bigr]$ consisting of all 
points in ${\mathbb S}_A$ whose first $M$ pairs of $N$-adic digits are fixed) has $H$-measure $\frac{1}{D^M},$ whereas when one pulls such sets back via $\Phi,$ we obtain cylinder sets  of the form  
\[
\bigl[(a_{j_1},b_{j_1})(a_{j_2},b_{j_2})\cdots (a_{j_M},b_{j_M}) \bigr] \subseteq X_A
\]
which also have measure ${D^{-M}}.$  Since these sets generate the Borel $\sigma$-algebras for ${\mathbb S}_A$ and $X_A$ respectively, we get the desired equality of the measures. 

Now let $\BB_D$ be the Bratteli diagram with $D$ vertices at each level, associated to {the matrix $A_D$ given in \eqref{matrixAD} (and, hence, to the directed graph $\Lambda_D$ with $D$ vertices and all possible edges)}. We equip the infinite path space $\partial \BB_D$ with the measure of Equation~\eqref{eq:measure}, which in this case is $M([\gamma]) = D^{-|\gamma|-1}$.
Label the vertices of $\Lambda^0$ by $\Z_D = \{0, 1, \ldots, D-1\}$, and define $\Theta : \partial \BB_D \rightarrow X_A$ by 
 \[
  \Theta((e_i)_{i \geq 1}) = \bigl( (a_{r(e_1)}, b_{r(e_1)}), (a_{r(e_2)}, b_{r(e_2)}), (a_{r(e_3)}, b_{r(e_3)}), \ldots)  \bigr);
 \] 
in other words, $\Theta$ takes an infinite path (written in terms of edges) $(e_i)_{i\in\N}$ to the sequence of vertices $(r(e_i))_{i\in\N}$ it passes through, and then maps this sequence of vertices to the corresponding element of $X_A$.
The map $\Theta$ is   bijective, since each pair of vertices has exactly one edge between them.  In addition, both $\Theta$ and $\Theta^{-1}$ are continuous, since both the topology on $\partial \BB_D$ and the topology on $X_A$ are generated by cylinder sets.  In other words, $\Theta$
is a homeomorphism, and $M=\mu_A \circ [\Theta]_*$.

We thus have shown that $\Upsilon= \Phi \circ \Theta$ is a Borel measure-theoretic isomorphism between the measure spaces $(\partial \BB_D, M)$ and $({\mathbb S}_A,H)$.  A routine computation, using the fact that 
\[ \Upsilon ((e_i)_{i\in \N}) = \left( \sum_{i\in \N} \frac{a_{r(e_i)}}{N^i}, \sum_{i\in \N} \frac{b_{r(e_i)}}{N^i} \right) ,\]
will show that for any $i \in \Z_D$, $T_i = S_i \circ \Upsilon$ to finish the proof.
\end{proof}

We now  recall the definition of Dutkay and Jorgensen \cite{dutkay-jorgensen-monic} of a {\it monic} representation of ${\mathcal O}_D:$

\begin{defn}[{cf.~\cite[Definition 2.6]{dutkay-jorgensen-monic}}]
\label{def-equiv-measures-O-D}
Let $D\in\mathbb N,$ and let $K_D$ be the infinite product Cantor group defined earlier. Let $\sigma_i:K_D\to K_D,\;0\leq i\leq D-1$ be 
as in Example \ref{ex:mp-fractal}.  A {\it nonnegative monic system} is a pair $(\mu, (f_i)_{i\in\mathbb Z_D})$ where $\mu$ is a Borel probability measure on $K_D$ and $(f_i)_{i\in\mathbb Z_D}$ are nonnegative Borel measurable functions in $L^2(K_D,\mu)$ such that  $\mu\circ \sigma_i^{-1} \ll \mu,$ and such that for all $i\in\mathbb Z_D$
$$\frac{d(\mu\circ \sigma_i^{-1})}{d\mu}=(f_i)^2$$
with the property that $f_i(x)\not=0,\;\mu$ a.e. on $\sigma_i(K_D),\;\forall i\in \Z_D.$  
\end{defn}
By Equation (2.9) of \cite{dutkay-jorgensen-monic}, there is a natural representation of ${\mathcal O}_D$ on $L^2(K_D, \mu)$ associated to a monic system  $(\mu, (f_i)_{i\in\mathbb Z_D})$ given by 
$$\tilde{S}_if\;=\;f_i(f\circ \sigma),\;(i\in\mathbb Z_D,\;f\in L^2(K_D, \mu)).$$
If $\mu=\nu_D$ 
and we set $f_i={\sqrt{D}}\chi_{\sigma_i(K_D)},$ the corresponding monic system is called the {\it standard positive monic system} for ${\mathcal O}_D.$ 

\begin{cor} 
\label{cor-equiv-measures-O-D} The representation of ${\mathcal O}_D$ on $L^2({\mathbb S}_A,H)$ described
in Section~\ref{sec:sierp-rep} above is equivalent to the monic representation of ${\mathcal O}_D$ corresponding  to the standard positive monic system on $L^2(K_D,\nu_D)$.   
\end{cor}

\begin{proof}
 Theorem \ref{thm:measure_preserving}, combined with the measure-theoretic identification of $(K_D, \nu_D)$ and $(\partial \BB_D, M)$ established in Example \ref{ex:mp-fractal}, implies that we have a measure-theoretic isomorphism between $(K_D, \nu_D)$ and $(\mathbb{S}_A, H)$.  Thus, to show that the corresponding representations of $ {\mathcal O}_D$ are unitarily equivalent, it only remains to check that the operators $\tilde{S}_i = f_i (f \circ \sigma)$ associated to the standard positive monic system, and the operators $\{T_i\}_{i \in \Z_D}$, match up correctly.  To that end, observe that 
 \begin{align*}
  \tilde{S}_i (\xi)(v_0 v_1 \ldots ) &  = f_i(v_0 v_1 \ldots ) \xi( v_1 v_2 \ldots)
           = \begin{cases} 
                 \sqrt{D} \xi(v_1 v_2 \ldots) & \text{ if } v_0 = i \\ 
                 0 & \text{ else.}
             \end{cases} \\
       & = S_i(\xi)(v_0 v_1 \ldots).
 \end{align*}
 Since Theorem \ref{thm:measure_preserving} established that the operators $S_i$ and $T_i$ are unitarily equivalent, the Corollary follows.
\end{proof}

\section{Spectral triples  and Laplacians for Cuntz algebras}
\label{sec:spect-triples}

Let $A_D$ be the $D\times D$ matrix with 1 in every entry and consider the  Bratteli diagram $\BB_D$ associated to $A_D$. If $D\ge 2$, then every row sum of $A_D$ is at least 2 by construction, and hence the associated infinite path space of the Bratteli diagram, $\partial \BB_D$, is a Cantor set. 
In this section, {by using the methods in} \cite{julien-savinien-transversal}, we will construct 
a spectral triple  on  $\partial \BB_D$. This spectral triple gives rise to a Laplace--Beltrami operator $\Delta_s$ on $L^2(\partial \BB_D, \mu_D)$, where $\mu_D$ is the measure induced from the Dixmier trace of the spectral triple as in Theorem \ref{thm-Dixmier-trace-cuntz-algebra-O_D} below. We also  compute explicitly the orthogonal decomposition of $L^2(\partial \BB_D, \mu_D)$ in terms of the eigenfunctions of the Laplace--Beltrami operator $\Delta_s$ (cf.~\cite[Theorem~4.3]{julien-savinien-transversal}).

\subsection{The Cuntz algebra $\mathcal{O}_D$ and its Sierpinski spectral triple}

\label{subsec:cuntz-spectral-triple}

\begin{defn}
\label{def-weight-paths}
Let $\Lambda$ be a finite directed graph; let $F(\BB_\Lambda)_\circ$ be the set of all finite paths on the associated Bratteli diagram, including the empty path whose length we set to $-1$ by convention..  A \emph{weight} 
 on $\BB_\Lambda$ (equivalently, on $\Lambda$) is a function $w: F(\BB_\Lambda)_\circ \to (0,\infty)$ satisfying 
\begin{itemize}
\item[(a)] $w(\circ) = 1$

\item[(b)]
\[
\lim_{n\to \infty} \sup \{ w(\eta): \eta \in \Lambda^n =  F^n\BB_\Lambda\} = 0,
\]
where we denoted by $\Lambda^n = F^n\BB_\Lambda$ the set of  finite paths of length $n $ on $\Lambda$ (equivalently, $\BB_\Lambda$).
\item[(c)] For any finite paths $\eta, \nu$ 
with $s(\eta) = r(\nu)$, we have $w(\eta \nu) < w(\eta) $.
\end{itemize}
A Bratteli diagram $\BB_\Lambda$ with a weight $w$ is called a \emph{weighted Bratteli diagram}.
\end{defn}

\begin{rmk}
\label{rmk-weights-on-vertices-and-edges} 
Observe that a weight that satisfies Definition 2.9 of \cite{julien-savinien-transversal}  on the vertices of a Bratteli diagram $\BB_\Lambda$ induces a  weight on the finite paths of the Bratteli diagram as in Definition \ref{def-weight-paths} above.  
In fact in \cite{julien-savinien-transversal} and \cite{pearson-bellissard-ultrametric} the authors define a weight on $F\BB_\Lambda$ by defining the weight first on vertices, and then extending it to finite paths via the formula  $w(\eta) = w(s(\eta))$,  for $\eta\in F\BB_\Lambda$.  
\end{rmk}

 We will show below that a weight on  $\BB_\Lambda$  induces in turn  a 
 measure on the 
 infinite path space $\partial \BB_\Lambda \cong \Lambda^\infty$; see Theorem \ref{thm-Dixmier-trace-cuntz-algebra-O_D} below for details.

\begin{defn}
An \emph{ultrametric} $d$ on a topological space $X$ is a metric satisfying the strong triangle inequality:
\[
d(x,y)\le \max\{d(x,z), d(y,z)\} \quad\text{for all $x,y,z\in X$.}
\]

\end{defn}

\begin{prop}[{\cite[Proposition~2.10]{julien-savinien-transversal}}]
\label{prop-on-weights}
Let $\BB_\Lambda$ be a weighted Bratteli diagram with weight $w$.  We define a function $d_w$ on $\partial B_\Lambda\times \partial B_\Lambda$ by
\[
d_w(x,y)=\begin{cases} w(x\wedge y) & \text{if $x\ne y$} \\ 0 & \text{otherwise}\end{cases},
\]
where $x\wedge y$ is the longest common initial segment of $x$ and $y$. (If $r(x) \not= r(y)$ then we say $x \wedge y $ is the empty path $  \circ$, and $w(\circ) = 1$.)
Then $d_w$ is an ultrametric on $\partial B_\Lambda$.
\end{prop}

Note that the ultrametric $d_w$ induces the same topology on $\partial \BB_\Lambda$ as the cylinder sets in \eqref{eq:cylin}; thus,  $(\partial \BB_\Lambda, d_w)$ is called an \emph{ultrametric Cantor set}.

\begin{defn}
\label{def-choice-of-weight-Cuntz-algebra}
Let $A_D$ be a $D\times D$ matrix with 1 in every entry and let $\BB_D$ be the associated Bratteli diagram. Fix $\lambda >1$, and set
\[
d = \ln D/\ln \lambda.
\] 
We define a weight $w_D^{\lambda}$ on the Bratteli diagram $\mathcal{B}_D$  by setting 
\begin{itemize}
\item[(a)] 
$
w_D^{\lambda}(\circ) = 1.
$
\item[(b)] For any level $0$ vertex $v \in V_0$ of $\BB_D$, $
w_D^{\lambda}(v) = \frac{1}{D}.
$
\item[(c)] For any finite path $\gamma \in F^n\mathcal{B}_D$ of length $n$,  
\[
w_D(\gamma) =\lambda^{-n} \frac1D.
\]
\end{itemize} \end{defn} 

According to \cite{julien-savinien-transversal}, after choosing a weight on $\BB_D$,  we can build a  spectral triple associated to it  as in the following Theorem. Note that this result  
is a special case of Section 3 of \cite{julien-savinien-transversal}. 

\begin{thm}\label{prop:spectral}
Fix an integer $D > 1$ and $\lambda >1 $.   Let $(\BB_D, w_D^{\lambda})$ be the weighted Bratteli diagram  with the choice of weight $w_D^{\lambda}$  as in Definition \ref{def-choice-of-weight-Cuntz-algebra}. Let $(\partial \BB_D, d_w^\lambda)$ be  the associated ultrametric Cantor set. Then there is an even spectral triple $(C_{\text{Lip}}(\partial \BB_D), \mathcal{H}, \pi_\tau, \slashed{D}, \Gamma)$,
where
\begin{itemize} 
\item $C_{\text{Lip}}(\partial \BB_D)$ is  the pre-$C^*$-algebra of Lipschitz continuous functions on $(\partial \BB_D, d_w^\lambda)$,
\item for each choice function $\tau: F\BB_D \to \partial \BB_D\times \partial \BB_D$,\footnote{A choice function $\tau:F\BB_D\to \partial \BB_D\times \partial \BB_D$ is a function that satisfies
\[
\tau(\gamma)= : (\tau_{+}(\gamma), \tau_{-}(\gamma))\quad\text{where}\quad d_w(\tau_{+}(\gamma), \tau_{-}(\gamma))=w_D^\lambda(\gamma).
\]
} a faithful representation $\pi_\tau$ of $C_{\text{Lip}}(\partial \BB_D)$ is given by bounded operators on the Hilbert space $\mathcal{H}=\ell^2(F\BB_D)\otimes \C^2$ as
\[
\pi_\tau(f)=\bigoplus_{\gamma\in F(\BB_D)_\circ}\begin{pmatrix} f(\tau_{+}(\gamma)) & 0 \\ 0 & f(\tau_{-}(\gamma))\end{pmatrix};
\]
\item the Dirac operator $\slashed{D}$ on $\mathcal{H}$ is given by
\[
\slashed{D}=\bigoplus_{\gamma\in F(\BB_D)_\circ} \frac{1}{w_D^\lambda(\gamma)}\begin{pmatrix} 0 & 1\\ 1 & 0\end{pmatrix};
\]
\item the grading operator is given by $\Gamma=1_{\ell^2(F(\BB_D)_\circ)}\otimes \begin{pmatrix} 1 & 0 \\ 0 & -1\end{pmatrix}$.
\end{itemize}
\end{thm}

\begin{defn}[{cf.~\cite[Theorem~3.8]{julien-savinien-transversal}}] The $\zeta$-function associated to the spectral triple of Theorem~\ref{prop:spectral} 
 is given by
\begin{equation}
\label{eq-def-zeta}
\zeta_D^\lambda(s) : ={\frac{1}{2}\Tr(|{\slashed{D}}|^{-s})}=
\sum_{\gamma\in F(\BB_D)_\circ}\big(w_D^\lambda(\gamma)\big)^s.
\end{equation}
\end{defn}

\begin{prop}[{cf.~\cite[Theorem~3.8]{julien-savinien-transversal}}]
\label{ab:conv-even-weight} 
The $\zeta$-function in Equation \eqref{eq-def-zeta} has abscissa of convergence equal to $d = \ln D / \ln \lambda $.
\end{prop}

\begin{proof}
By a straightforward calculation we get  (if we denote by $F^q(\BB_D)_\circ$ the set of paths of length $q$):
 \[
  \sum_{\gamma \in F(\BB_D)_\circ} \Big(w_D^\lambda(\gamma)\Big)^s 
      = D^{-s} \sum_{q \geq -1} \Card (F^q(\BB_D)_\circ) \lambda^{-qs}
      = D^{-s} \sum_{q \geq -1} D^{q+1} \lambda^{-qs},
 \]
 where $\Card(S)$ denotes the cardinality of the set $S$.
 It is clear that this sum converges  precisely when $D/\lambda^s$ is smaller than~$1$, that is whenever 
 \[
  s > \frac{\ln D}{\ln \lambda}.
 \]
 \end{proof}

It is known that the abscissa of convergence coincides with the upper Minkowski dimension of $\partial \BB_D \cong \Lambda^\infty_D$ associated to the ultrametric $d_w^\lambda$~\cite[Theorem~2]{pearson-bellissard-ultrametric}. In the self-similar cases (when the weight is given as in Definition~\ref{def-choice-of-weight-Cuntz-algebra}), the upper Minkowski dimension turns out to coincide with the Hausdorff dimension~\cite[Theorem~2.12]{julien-savinien-embedding}.
In particular, when the scaling factor $\lambda$ is just $N$, the Hausdorff dimensions of $(\Lambda^\infty_D, d_{w_D^N})$ and $\mathbb{S}_A$ coincide, where we equip $\mathbb{S}_A$ with the metric induced by the Euclidean metric on $[0,1]^2$.

The Dixmier trace $\mu_D^\lambda(f)$ of a function $f\in C_{\text{Lip}}(\partial \BB_D)$ is given by the expression below; see Theorem 3.9 of \cite{julien-savinien-transversal} for details. 
\begin{equation}
\label{eq:zeta-mu}
\mu_D^\lambda(f)=\lim_{s\downarrow d}\frac{\Tr(|{\slashed{D}}|^{-s}\pi_\tau(f))}{\Tr(|{\slashed{D}}|^{-s})}=\lim_{s\downarrow d}\frac{\Tr(|{\slashed{D}}|^{-s}\pi_\tau(f))}{ 2 \zeta_D^\lambda(s)  } .
\end{equation}

In particular the  limit given in \eqref{eq:zeta-mu} induces a measure $\mu_D^\lambda$ on $\partial \BB_D$ characterized as follows.  If $f=\chi_{[\gamma]}$ is the characteristic function of a cylinder set $[\gamma]$, and if $F_\gamma\BB_D = \{ \eta \in F_\gamma\BB_D: \eta = \gamma \eta'\}$ denotes the set of all finite paths with initial segment $\gamma$, we have

\begin{equation}\label{eq:m_ind_Dix}
\mu_D^\lambda([\gamma]) = \mu_D^\lambda(\chi_{[\gamma]})
                =\lim_{s\downarrow d} \frac{\sum_{\eta\in F_\gamma(\BB_D)_\circ}\big(w_D^\lambda(\eta)\big)^s}{\sum_{\eta\in F(\BB_D)_\circ}\big(w_D^\lambda(\eta)\big)^s}.
\end{equation}

It actually turns out, as we prove in Theorem \ref{thm-Dixmier-trace-cuntz-algebra-O_D} below, that the measure $\mu_D^\lambda$ on $\partial \BB_D$ 
is independent of $\lambda$; so we will also write, with notation as above
\[
\mu_D([\gamma])=\mu_D(\chi_{[\gamma]}) = \mu_D^\lambda([\gamma])=\mu_D^\lambda(\chi_{[\gamma]})
\]

Moreover,
by combining Theorem \ref{thm:measure_preserving} with    Theorem \ref{thm-Dixmier-trace-cuntz-algebra-O_D} below, 
we see that $\mu_D$ agrees with the Hausdorff measure of $\mathbb{S}_A$.} 

\begin{thm}
\label{thm-Dixmier-trace-cuntz-algebra-O_D}
For any choice of scaling factor $\lambda >1$, the measure $\mu_D^\lambda$ on $\partial \BB_D$ {induced by the Dixmier trace} 
agrees with the measure $M$ associated to the infinite path representation of $\mathcal{O}_D$. 
Namely, for  any finite path  $\gamma \in F\BB_D$, we have
\begin{equation} 
\mu_D([\gamma])=\frac{1}{D^{|\gamma|+1 }} = 
M([\gamma]).
 \end{equation}
 \end{thm}

 \begin{proof}  Note that, although  the proof of this  Theorem  is very long for the more general case of Cuntz--Krieger algebras (cf.~\cite[Theorem 3.9]{julien-savinien-transversal}), 
 it considerably simplifies for the case of  Cuntz algebras covered here.
  First note that for the choice of the empty path $\gamma= \circ  $ (whose cylinder set corresponds to the whole space), we have
\begin{equation}\label{limit-eq-cuntz-case-circ}
  f(s) = \frac{\sum_{\eta \in F(\BB_D)_\circ}  (w_D^\lambda(\eta))^s} {\sum_{\eta \in F(\BB_D)_\circ}  (w_D^\lambda(\eta))^s } =1=\mu_D^\lambda(\Lambda^{\infty}_D) =M(\Lambda^{\infty}_D ) . 
\end{equation}
Now we will compute $\mu_D^\lambda$, for  a finite path $\gamma\not= \circ $ of length $n$ in $F^n\BB_D$. 
Define, according to  Equation \eqref{eq:m_ind_Dix},
\begin{equation}\label{limit-eq-cuntz-case}
  f(s) = \frac{\sum_{\eta \in F_\gamma \BB_D}  (w_D^\lambda(\eta))^s} {1+ \sum_{\eta \in F\BB_D}  (w_D^\lambda(\eta))^s }. 
\end{equation}

Note that in the above expression we isolated the term corresponding to the empty path, for which $\Big( w^\lambda_D(\circ)\Big)^s = 1^s =1$. Moreover, since  $\gamma$ is not the empty path, then  $\eta = \circ$ does not occur in the sum in the numerator.
 If $\eta\in F_\gamma\BB_D$, then $w_D^\lambda(\eta)^s$ only depends on the length of $\eta$, say $|\eta| = n+q$ for some $q\in \N_0$, and hence $w_D^\lambda(\eta)=D^{-1} \lambda^{-(n+q)}$. 
 For $q\in \N_0$, let
 \[\begin{split}
    F^q\BB_D        & = \{ \eta \in F\BB_D        : \vert \eta \vert = q   \}, \\
    F_\gamma^q\BB_D & = \{ \eta \in F_\gamma\BB_D : \vert \eta \vert = n+q \}.
 \end{split}\]
 Then we can write 
  \[
    f(s) = \frac{D^{-s} \sum_{q\in \N_0} \Card(F_\gamma^q\BB_D) \big( \lambda^{-(n+q)} \big)^s}{
                 1 + D^{-s}\sum_{q\in \N_0} \Card(F_\gamma^q\BB_D)        \big( \lambda^{-q}   \big)^s}.
  \]
 Since the vertex matrix $A_D$ of the Bratteli diagram $\BB_D$ has 1 in every entry, every edge in $\BB_D$ has $D$ possible edges that could follow it. Also note that $\eta\in F^q\BB_D$ has its range in $V_0$ and its source in $V_q$, and hence we get 
 \[
   \Card(F^q\BB_D)=D^{q+1}.
 \]
 But  any finite path $\eta\in F^q_\gamma\BB_D$ can be written as $\eta=\gamma\eta'$. Since  $\gamma$ is fixed,  the number of  paths $\eta \in F^q_\gamma\BB_D$ is the same as the number of possible paths $\eta'$. Since $r(\eta')=s(\gamma)$ is also fixed, we get
 \[
   \Card(F^q_\gamma\BB_D)= D^q.
 \]
 
By multiplying both numerator and denominator of $f(s)$ by $D^s$, we obtain
  \begin{align*}
  f(s)& =\frac{D^{-s} \sum_{q\in \N_0} D^{q} \big({\lambda^{-(n+q)}}\big)^s}{ 1+ D^{-s}\sum_{q\in \N_0} D^{q+1} \big({\lambda^{-q}}\big)^s} = \frac{1}{\lambda^{ns}}\frac{\sum_{q\in \N_0}D^q \lambda^{-qs}}{D^s+\sum_{q\in \N_0} D^{q+1}\lambda^{-qs}} \\
 & =\frac{1}{\lambda^{ns}}\frac{\sum_{q\in \N_0}\Big(\frac{D}{\lambda^s}\Big)^q}{\left(D^s+D\sum_{q\in \N_0}\Big(\frac{D}{\lambda^s}\Big)^q\right)}.
  \end{align*}
  Since $s>\frac{\ln D}{\ln \lambda}$, we have $\frac{D}{\lambda^s}<1$, thus $\sum_{q\in \N_0}\Big(\frac{D}{\lambda^s}\Big)^q$ converges and is equal to $\frac{1}{1-\frac{D}{\lambda^s}}$. Thus (again multiplying numerator and denominator of $f(s)$ by $1-\frac{D}{\lambda^s}$),
 \[
 f(s)=\frac{1}{\lambda^{ns}}\frac{\frac{1}{1-\frac{D}{\lambda^s}}}{(D^s+D\frac{1}{1-\frac{D}{\lambda^s}})}
 =\frac{1}{\lambda^{ns}}\frac{1}{\Big((1-\frac{D}{\lambda^s})D^s+D\Big)}
 \] 
  Now take the limit $s\downarrow d$ and recall that $\lambda^d=D$. So we have $(1-\frac{D}{\lambda^s})\to 0$ and hence
 \[
 \lim_{s\downarrow d}f(s)=\frac{1}{\lambda^{nd}}\frac{1}{D}=\frac{1}{D^n}\frac{1}{D}=\frac{1}{D^{n+1}},
 \]
 which is the desired result by Equation \eqref{eq:m_ind_Dix}.
 \end{proof}

 \subsection{The Laplace--Beltrami operator}
 \label{subsec-Laplace-Beltrami Operator-O-D}
 
 In Section 4 of \cite{julien-savinien-transversal}, the authors use the spectral triple associated to a weighted Bratteli diagram to construct a non-positive definite self-adjoint operator with discrete spectrum (which they fully describe) defined on the infinite path space of the given Bratteli diagram. Moreover, they show in Theorem 4.3 of  \cite{julien-savinien-transversal} that the eigenfunctions of $\Delta_s$ form an orthogonal decomposition of the $L^2$-space of the boundary.
 
 Therefore, by  applying the results of Section 4 of \cite{julien-savinien-transversal} to the spectral triples of  Section \ref{subsec:cuntz-spectral-triple} above, we obtain, after we choose a weight $w_D^{\lambda}$ on $\BB_D$ as in Definition \ref{def-choice-of-weight-Cuntz-algebra},  a non-positive definite self-adjoint operator $\Delta_s$ on $L^2(\partial \BB_D, \mu_D)$ for any $s \in \R$,  where $\mu_D$ is the measure on $\partial \BB_D$ given in \eqref{eq:m_ind_Dix}. (Recall that $\mu_D$ does not depend on $\lambda$). 
 Namely, for any $s\in \R$, the Laplace--Beltrami operator $\Delta_s$ on $L^2(\partial \BB_D, \mu_D)$
  is given by
  \begin{equation}\label{eq:Delta}
  \langle f, \Delta_s(g)\rangle=Q_s(f,g)=\frac{1}{2}\int_E \Tr(\vert \slashed{D}\,\vert^{-s}[\slashed{D}, \pi_\tau(f)]^*\,[{\slashed{D}},\pi_\tau(g)]\, d\mu_D(\tau),
  \end{equation}
  where $\Dom Q_s= \Span  \{ \chi_{[\gamma]} : \gamma\in F\BB_D \}$ and $Q_s$ is a closable Dirichlet form, and $\mu_D(\tau)$ is the measure induced by the Dixmier trace on the set $E$ of choice functions.

 Moreover, 
  the eigenfunctions of $\Delta_s$ form an orthogonal decomposition of $L^2(\partial \BB_D, \mu_D)$. In the remainder of this section we give the details of this decomposition and formulas for the  eigenvalues.   In Section \ref{wavelets-and-eigenfunctions-O-D} below, we describe the relationship between this orthogonal decomposition and the wavelet decomposition of $L^2(\Lambda^\infty, M)$ computed in \cite{FGKP}.

   \begin{thm}\cite[Theorem~4.3]{julien-savinien-transversal}\label{thm:eigen}
  Let $(\mathcal{B}_D, w_D^{D})$ be the weighted  Bratteli diagram as in Theorem \ref{prop:spectral}. (Note that we made here the choice  $\lambda=D$  for simplicity.) 
 Let $\Delta_s$ be the Laplace--Beltrami operator on $L^2(\partial \BB_D, \mu_D)$ given by \eqref{eq:Delta}.
 Then  the eigenvalues of $\Delta_s$ are $0$, associated to the constant function $1$, and the eigenvalues $\{\lambda_\eta\}_{\eta \in (F\BB_D)_\circ}$ with corresponding eigenspaces $\{E_\eta\}_{\eta \in (F\BB_D)_\circ}$ of $\Delta_s$ are given by
 \[
  \lambda_\circ = \bigl( G_s (\circ) \bigr)^{-1} = \frac{2D}{D-1};
 \]
 \[
  \lambda_\eta = -2 -2D^{3-s} \frac{1-D^{(3-s)|\eta|}}{1-D^{3-s}} - \frac{2D^{3 |\eta| +4}}{(D-1)D^{s(|\eta| +1)}}, \quad \eta \in F(\BB_D)
 \]
 with
 \[
  E_\circ = \Span \Bigl \{ D^{-1} \bigl( \chi_v   - \chi_{v'} \bigr)   \Bigr\} \ : \ v \ne v' \in V_0 \Bigr\},
 \]
 \begin{multline*}
  E_\eta= \Span \Big\{\, D^{|\eta|+2} \left( {\chi_{[\eta e]}}-{\chi_{[\eta e']}} \right)\; : \\ \eta \in F(\BB_D),\  e \ne e',\; |e|=|e'|=1,\; r(e)=r(e') = s(\eta)\, \Big\}.
 \end{multline*}
\end{thm}
\begin{proof}
This follows from evaluating the formulas given in Theorem 4.3 of \cite{julien-savinien-transversal}, using Theorem \ref{thm-Dixmier-trace-cuntz-algebra-O_D} above to calculate the measures of the cylinder sets, and recalling that the diameter $\text{diam}[\gamma]$ of a cylinder set is given by the weight of $\gamma$.  

To be precise, since there are $D$ edges with a given range $v$, the size of the set 
\[
 \bigl\{ (e, e') \in \Lambda^1 \times \Lambda^1: r(e) = r(e') = v, \ e \not= e' \bigr\}
\]
is $D(D-1)$ for any vertex $v$.  Therefore, for any path $\eta \in \Lambda$, the constant $G_s(\eta)$ from Theorem 4.3 of \cite{julien-savinien-transversal} is given by
\[
 G_s(\eta) = \frac{D(D-1) D^{-2(|\eta| + 2)}}{2w_D(\eta)^{s-2}} = \frac{(D-1) D^{s(|\eta| +1)}}{2D^{4|\eta| + 5}}.
\]
Observe that, in the notation of \cite{julien-savinien-transversal}, a path of ``length 0'' corresponds to the empty path $\circ$, that is, whose cylinder set gives entire infinite path space, and a path of ``length 1'' corresponds to a vertex.  In general, the length of a path in  \cite{julien-savinien-transversal} corresponds to the number of vertices that this path traverses; hence a path of length $n$ for them is a path of length $n-1$ for us.  

In order to compute the eigenvalues $\lambda_\eta$ described in Theorem 4.3 of \cite{julien-savinien-transversal}, then, we also need to calculate $G_s(\circ) = G_s(\Lambda^\infty)$.  Since the infinite path space has diameter 1 by Proposition 2.10 of \cite{julien-savinien-transversal}, we obtain 
\[ G_s(\circ) = G_s(\Lambda^\infty) = \frac{D(D-1)}{2 D^2} = \frac{D-1}{2D}.\]

Now, if we denote the empty path $\circ$ by a path of ``length $-1$,'' we can rewrite the formula (4.3) from \cite{julien-savinien-transversal} for the eigenvalue $\lambda_\eta$ associated to a path $\eta$ as 
\begin{align*}
\lambda_\eta & = \sum_{k=-1}^{|\eta| -1} \frac{\frac{1}{D^{k+2}} 
               - \frac{1}{D^{k+1}}}{G_s(\eta[0,k])}
               - \frac{1}{D^{|\eta|+1} G_s(\eta)} \\
             & = \frac{1-D}{D\frac{D-1}{2D}}
               + \sum_{k=0}^{|\eta|-1} \frac{1-D}{D^{k+2}} \frac{2 D^{4k + 5}}{(D-1)D^{s(k+1)}}
               - \frac{2D^{3|\eta| + 4}}{(D-1)D^{s(|\eta| +1)}} \\
             & = - 2 - 2D^{3-s} \frac{1-D^{(3-s)|\eta|}}{1-D^{3-s}}
               - \frac{2D^{3|\eta| + 4}}{(D-1)D^{s(|\eta| +1)}},
\end{align*}
using the notation of Definition \ref{def:bratteli-diagram}, and the fact that 
\[\sum_{k=0}^{|\eta|-1}  \frac{2 D^{3k + 3}}{D^{s(k+1)}} =2D^{3-s} \sum_{k=0}^{|\eta|-1} \frac{D^{3k}}{D^{sk}} = 2D^{3-s} \frac{1-D^{(3-s)|\eta|}}{1-D^{3-s}}.\]
\end{proof}

\section{Wavelets and eigenfunctions for $\mathcal{O}_D$}
\label{wavelets-and-eigenfunctions-O-D}

In this section, we connect the eigenspaces $E_\gamma$ of Theorem \ref{thm:eigen} with the orthogonal decomposition of $L^2(\partial \BB_D, M)$ associated to the wavelets constructed in \cite{MP} Section 3 (see also Section 4 of \cite{FGKP}).  We begin by describing the wavelet decomposition of $L^2(\partial \BB_D, M)$, which is a special case of the wavelets of \cite{MP}  and \cite{FGKP}.
To be precise, the wavelets we discuss here are those associated to the $D \times D$ matrix $A_D$ consisting of all 1's, but the wavelets described in \cite{MP} are defined for any matrix $A$ with entries from $\{0,1\}$.

Let $\Lambda_D$ denote the directed graph with vertex matrix $A_D$.  In what follows, we will assume that we have labeled the $D$ vertices of $\Lambda_D^0$ by $\Z_D =  \{ 0, 1, \ldots, D-1\} ,$
 and we will write infinite paths in $\Lambda^\infty_D = \partial \BB_D$ as strings of vertices $(i_1 i_2 i_3 \ldots )$ where $i_j \in \Z_D$ for all $j$.

Denote by ${\mathcal V}_0$ the (finite-dimensional) subspace of $L^2(\partial \BB_D, M)$ given by
\[
 {\mathcal V}_0 \;=\; \Span \{\chi_{\sigma_i(\partial \BB_D)}:\;i\in\mathbb Z_D\}.
\]
Define an inner product on $\mathbb C^{D}$ by
\begin{equation}\label{eq:inner-prod}
  \bigl\langle (x_j), (y_j) \bigr\rangle_{PF}\;=\;\frac{1}{D}\sum_{j=0}^{D-1}\overline{x_j}y_j.
\end{equation}
We now define a set of $D$ linearly independent  vectors $\{c^{j}:\;0\leq j \leq D-1\}\subset \mathbb C^{D}$, where $c^{j} = (c^{j}_0, \ldots, c^{j}_{D -1})$, by
\[
  c_{\ell}^{0}= 1\;\; \forall \ \ell \in \Z_D,
\]
and 
$\{c^{j}: 1\leq j\leq D-1\}$  an orthonormal basis for the  subspace $\{(1,1,\cdots,1)\}^{\perp}$, with  ${\perp} $ taken with respect to the inner product $\langle \cdot,\cdot \rangle_{PF}$.

We now note that we can write each set $R_{[k]} = \sigma_k (\partial \BB_D)$ 
as a disjoint union:
$$R_{[k]}=\bigsqcup_{j=0}^{D-1}R_{[kj]},$$
where
\[
 R_{[kj]}\;=\; \bigl\{ (i_1i_2\cdots i_n\cdots )
   \in \partial \BB_D:\;\;i_1=k\;\text{and}\;i_2=j \bigr\}.
\]
Thus in terms of characteristic functions,
$$\chi_{R_{[k]}}\;=\;\sum_{j=0}^{D-1}\chi_{R_{[kj]}}\;\;\text{for $k\in\mathbb Z_D$}.$$
Now, define functions $\{f^{j,k}\}_{j, k=0}^{D-1}$  on $\partial \BB_D$ by
$$f^{j, k}(x)\;=\;\sqrt{D}\sum_{\ell=0}^{D-1}c_{\ell}^{j}\chi_{R_{[k{\ell}]}}(x).$$
Moreover,  since $c_\ell^0 = 1$ for all $\ell$, we have 
$$f^{0,k}=\sqrt{D}\sum_{\ell=0}^{D-1}c_{\ell}^{0}\chi_{R_{[k{\ell}]}} = \sqrt{D}\sum_{\ell=0}^{D-1} \chi_{R_{[k\ell]}} =  \sqrt{D} \chi_{R_{[k]}}.$$
It follows that
\[
 \Span \bigl\{ f^{0,k} \bigr\}_{k=0}^{D-1} \;=\; \Span \bigl\{ \chi_{R_{[k]}} \bigr\}_{k=0}^{D-1}\;=\;{\mathcal V}_0.
\]

Now, we can use the functions $f^{j,k}$ to construct a wavelet basis of $L^2(\partial \BB_D, M)$.  First, a definition: for any word $w = w_1 w_2 \cdots w_n \in (\Z_D)^n$, write $S_w = S_{w_1} S_{w_2} \cdots S_{w_n}$, where $S_{w_i} \in L^2(\partial \BB_D, M)$ is the operator defined in Proposition~\ref{MPrepprop}.
\begin{thm}[{\cite[Theorem 3.2]{MP}; \cite[Theorem 4.2]{FGKP}}]
\label{MPwaveletsthm}
Fix an integer $D>1.$  Let $\{S_k\}_{k\in \Z_D}$ be the operators on $L^2(\partial \BB_D, M)$ described in Proposition~\ref{MPrepprop}. Let $\{f^{j,k}:\;j,k\in\mathbb Z_D\}$ be the functions on $\partial \BB_D$ defined in the above paragraphs. Define
$${\mathcal W}_0\;=\; \Span \{f^{j,k}: j, k\in \mathbb Z_D, j \not= 0\};$$
$${\mathcal W}_n= \Span \{S_w(f^{j,k}):\;j,k\in \mathbb Z_D, \ j \not= 0, \text{ and}\;\;w \in (\Z_D)^n\}.$$
Then the subspaces ${\mathcal V}_0$ and $\{{\mathcal W}_n\}_{n=0}^{\infty}$ are mutually pairwise orthogonal in $L^2(\partial \BB_D, M)$ and
$$L^2(\partial \BB_D, M)=\; \Span \left({\mathcal V}_0\oplus \Big[\bigoplus_{n=0}^{\infty} \mathcal{W}_n\Big]\right).$$
\end{thm}

To calculate the functions $S_w (f^{j,k})$, we first observe that 
\[ S_i \chi_{R_{[k]}} = \sqrt{D} \chi_{R_{[ik]}};\]
 consequently, if $w = w_1 w_2 \cdots w_n$,
\begin{equation}
\label{eq:wavelet-fcns}
S_w(f^{j,k}) = D^{(n+1)/2} \sum_{\ell = 0}^{D-1} c^j_\ell \chi_{[w_1 w_2 \cdots w_n k \ell]}.
\end{equation}
If we instead write the finite path $w $ as $ \gamma$, and observe that the edges in $E_{n+1}$ with range $k$ are in bijection with the pairs $(k\ell)_{\ell \in \Z_D}$, we see that for any path $\gamma \in F\BB_D$ with $|\gamma | = n-1$,
\begin{equation}
\label{eq:wavelet-fcns-2}
S_\gamma (f^{j, k}) = D^{(n+1)/2} \sum_{e \in E_{n+1}} c^j_e \chi_{[\gamma k e]}.
\end{equation}
A few more calculations lead us to the following 
\begin{thm}
\label{equalweightswaveletOD}
Let $\Lambda_D$ be the directed graph whose $D \times D$ adjacency matrix consists of all 1's.  For each $\gamma \in \Lambda$, let $E_\gamma$ be the eigenspace of the Laplace--Beltrami operator $\Delta_s$ described in Theorem  \ref{thm:eigen}.
Then 
for all $n \geq 0$ we can write
\[
 \mathcal W_{n} = \bigoplus_{\gamma \in 
  {{}\Lambda^{n}}} E_\gamma.
\]
In particular,
\[
 L^2(\Lambda^\infty, \mu) = \mathcal V_{-1} \oplus 
  {\mathcal W}_{-1} \oplus \Biggl[ \, \bigoplus_{n \geq 0} \bigoplus_{\gamma \in \Lambda^{n}} E_\gamma   \Biggr] =  {\mathcal V_{0}\oplus \Biggl[ \, \bigoplus_{n \geq 0} \bigoplus_{\gamma \in \Lambda^{n}} E_\gamma   \Biggr].}
\]
Moreover, for all $i \in \Z_D$ and all $\gamma \in F\BB_D$,, the isometry $S_i$ given by 
\[
 {S}_i f((v_1 v_2\ldots )) = \begin{cases}  D^{1/2} f  ((v_2 v_3 \ldots))& \text{if } v_1 = i, \\
          0 & \text{else.}
          \end{cases}
\]
maps $E_\gamma$ to $E_{i \gamma}$ unitarily.
\end{thm}

\begin{proof}
Let $\gamma \in F\BB_D$ be a path of length $n$.
Recall that the subspaces $E_\gamma$ are spanned by functions of the form $\chi_{[\gamma e]} - \chi_{[\gamma e']}$, where $e \not= e'$ are edges in $E_{n+1}$.  In other words, if we write a spanning function $\xi_{e, e'} = \chi_{[\gamma e]} - \chi_{[\gamma e']}$ of $E_\gamma$ as a linear combination of characteristic functions of cylinder sets, we have 
\[ \xi_{e, e'} = \sum_{f\in E_{n+1}} d_f \chi_{[\gamma f]}\]
where $d_e = 1, d_{e'} = -1, d_f = 0 \ \forall \ f \not= e, e'$.  In other words, the vector 
\[(d_f)_{r(f) = s(\gamma), f\in E_{n+1}}\]
 is in the subspace $(1, 1, \ldots, 1)^\perp$ of $\C^D$ which is orthogonal to $(1,1, \ldots, 1)$ in the inner product \eqref{eq:inner-prod}.  It follows that $E_\gamma \subseteq \mathcal{W}_n$ whenever $|\gamma| = n$.
 
 Now, Theorem 4.3 of \cite{julien-savinien-transversal} tells us that each space $E_\gamma$ has dimension $D-1$.  Moreover,  there are $D^{n+1}$  paths $\gamma$ of length $n$, and $E_\gamma \perp E_\eta$ for all $\gamma, \eta$ with $|\gamma| = |\eta|$.  Therefore, 
 \[ \dim \left( \bigcup_{|\gamma| = n} E_\gamma\right) = {D^{n+1}(D-1)}.\]
 Similarly, $\dim \mathcal{W}_n = \text{Card} (F^{n-1}\BB_D) \text{Card} ( \{ f^{j,k}\}_{j\not= 0}) = {D^n \cdot D(D-1)}{}$.  This equality of dimensions thus implies that 
 \[ \mathcal{W}_n = \bigcup_{|\gamma| = n} E_\gamma \ \forall \ n \in \N_0.\]
 
 For the last assertion, we simply observe that $S_i$ is an isometry with $S_i S_i^* = id|_{E_i}$.
\end{proof}

\subsection{Wavelets on $\mathbb{S}_A$}
Let $A$ be an $N \times N$ $\{0, 1\}$-matrix with precisely $D$ nonzero entries.
In this section we will describe wavelets on $\mathbb{S}_{A}$ associated to the Cuntz algebra $\mathcal{O}_D$ using the measure-preserving isomorphism between $(\mathbb{S}_A, H)$ and $(\partial \BB_D, M)$ described in Theorem~\ref{thm:measure_preserving}.

Since all edges in $\Lambda_D$ can be preceded (or followed) by any other edge, this infinite path space corresponds simply to $[0,1]$ by thinking of points in $[0,1]$ as infinite sequences in $\{0, \ldots, D-1\}^{\N}$ and using the $D$-adic expansion. 
  
The natural correspondence between $\mathbb{S}_A$ and points from $[0,1]$ in their $D$-adic expansions is given by labeling the nonzero entries in $A$ by the elements of $\{0, 1, \ldots, D-1\}$, and then identifying a cylinder set 
$[(x_1,y_1),(x_2,y_2),\dots,(x_n,y_n)]$ in $\mathbb{S}_A$ with the cylinder 
 $[d_1\dots d_n]$, where $d_i \in \{0, \ldots, D-1\}$ is the integer corresponding to $A_{x_i,y_i}$.
  
Thus, we obtain wavelets on $\mathbb{S}_A$ by using this identification to transfer the wavelets associated to the infinite path representation of $\mathcal{O}_D$ into functions on $\mathbb{S}_A$.  These wavelets will agree with the eigenfunctions $E_\gamma$ of the Laplace--Beltrami operator associated to the Bratteli diagram for $\mathcal{O}_D$, by Theorem \ref{equalweightswaveletOD} above. 

To be more precise, Theorem \ref{equalweightswaveletOD} implies that we can interpret the eigenfunctions of Theorem \ref{thm:eigen} as a 
wavelet decomposition of $L^2(\partial \BB_D, M)$, with 
\[
 E_\gamma= \Span \biggl\{  \frac{1}{M[\gamma e]} \chi_{[\gamma e]} - \frac{1}{M[\gamma e']} \chi_{[\gamma e']} \biggr\}.
\]
Here $\gamma$ is a finite path in the graph $\Lambda_D$ associated to $\mathcal{O}_D$; writing $\gamma$ as a string of vertices, equivalently, $\gamma =d_0 d_1 d_2 \cdots d_n$ for $d_i \in \{ 0, \ldots, D-1\}$.  Thus, if $d_i \in \Z_D$ corresponds to the pair $(x_i, y_i) \in S_A$, and $e, e' \in \{0, \ldots, D-1\}$ correspond to the pairs $(z, w), (z',w')$ in {{}the symbol set} $S_A$, the wavelet on $L^2(\mathbb{S}_A, H)$ associated to $\ \frac{1}{M[\gamma e]} \chi_{[\gamma e]} - \frac{1}{M[\gamma e']} \chi_{[\gamma e']}$ is 
\begin{multline*}
 \frac{1}{H([(x_1,y_1),(x_2,y_2),\dots,(x_n,y_n), (z, w)])} \chi_{[(x_1,y_1),(x_2,y_2),\dots,(x_n,y_n), (z,w)]}\\
 - \frac{1}{H([(x_1,y_1),(x_2,y_2),\dots,(x_n,y_n), (z', w')])} \chi_{[(x_1,y_1),(x_2,y_2),\dots,(x_n,y_n), (z',w')]} \\
 	= \frac{1}{D^{n+2}}\left(  \chi_{[(x_1,y_1),(x_2,y_2),\dots,(x_n,y_n), (z, w)]} -  \chi_{[(x_1,y_1),(x_2,y_2),\dots,(x_n,y_n), (z', w')]} \right). 
\end{multline*}
This correspondence allows us to transfer the spaces $E_\gamma$ from $L^2(\partial \BB_D, M)$ to $L^2(\mathbb{S}_A, H)$, giving us an orthogonal decomposition of the latter.  Moreover, the ``scaling and translation'' operators $S_i$ of Theorem \ref{equalweightswaveletOD} from the infinite path representation of $\mathcal{O}_D$ transfer (via the same correspondence between pairs ($x, y)$ with $A({x,y}) \not= 0$ and elements of $\{0, \ldots, D-1\}$) to the operators $T_i$  on $L^2(\mathbb{S}_A, H)$ introduced in Theorem \ref{thm:measure_preserving}.  In other words, these operators $T_i$ allow us to move between the orthogonal  subspaces of $L^2(\mathbb{S}_A, H)$, enabling us to view this as a wavelet decomposition.

\section{Spectral triples  and Laplacians for the Cuntz algebra $\mathcal{O}_D$: the uneven weight case}
\label{sec:spect-triples-O-2}

\subsection{The spectral triple}

We are going to work in the general framework of Section~\ref{sec:spect-triples} with the difference that the weight (which we call $w_D^r$) is different from the Perron--Frobenius weights  $w_D^\lambda$ we previously defined in Definition \ref{def-choice-of-weight-Cuntz-algebra}.
For this section, 
we require that our weight is  defined on finite paths as in Definition \ref{def-weight-paths}, rather than on vertices as in Definition \ref{def-choice-of-weight-Cuntz-algebra}. 
In particular, the weight $w_D^r$ will not be self-similar in the sense that $w_D^r(\gamma)$ will not depend only on the length and the source of $\gamma$, but also on the precise sequence of edges making up $\gamma$.

\begin{defn}
\label{def-choice-of-weight-Cuntz-O-D-algebra}
 Fix a vector $r = (r_1, \ldots, r_D)$ of positive numbers satisfying $\sum_i r_i =1$. (We also note that this condition is not essential, although it makes a nice normalization.)
The weight $w_D^r$ on the graph $\Lambda_D$  with $D$ 
vertices $v_1, \ldots, v_D$ (equivalently, the Bratteli diagram $\mathcal{B}_D$) associated 
to the matrix $A_D$  is defined as follows.
\begin{enumerate}
\item Whenever $\gamma$ is the trivial (empty) path $\circ$, we set $w_D^r(\circ) =1$.
\item {Associate to each vertex $v_i$ the weight $r_i$:}
\[
 w^r_D (v) =  r_v,\ \forall v \in \Lambda_0.
\]
\item  Given a path $\gamma = (e_1 \ldots e_n)$  with $|e_j|=1$, $s(e_i) = v_{j_i}$, and $r(e_1) = v_{j_0}$, we set the weight of $\gamma$ to be
\[
 w^r_D (\gamma)  = \prod_{i=0}^n r_{j_i}  .
\]
\item The diameter $\diam [\eta]$ of a cylinder set $[\eta]$
is defined to be equal to its weight,
\[
\diam [\eta] = w^r_D (\eta) .
\] 
\end{enumerate}
\end{defn}

Note in particular that $[\circ]= \Lambda_D^{\infty}$ and so $diam[\circ]=1$, which is
consistent with the choice of our normalization.

The set of finite paths on a graph has a natural tree structure.
{In fact, if  we denote by  $(e_1 \ldots e_n)$ a string of composable edges (thus requiring $s(e_{i-1}) = r(e_i),\  \forall i$)} then  the ``parent'' of $(e_1 \ldots e_n)$  is $(e_1 \ldots e_{n-1})$; the root is the path $\circ$ of length $-1$ which corresponds to $\Lambda^\infty_D$.
In addition, the weight $w_D^r(\gamma)$ decreases to $0$ as the length of $\gamma$,  $|\gamma|$, increases to infinity.
Therefore, the Pearson--Bellissard construction from~\cite{pearson-bellissard-ultrametric} applies, and there is a spectral triple associated to the set of infinite paths as in Theorem~\ref{prop:spectral} (see also~\cite{julien-savinien-transversal,FGKJP-spectral-triples}).  To be more precise, we have:

\begin{prop}\label{prop:spectral-uneven}
Let $\mathcal{B}_D$ be the Bratteli diagram associated to the matrix $A_D$.
Let $(\BB_D, w_D^r)$ be the weighted Bratteli diagram given in Definition~\ref{def-choice-of-weight-Cuntz-O-D-algebra}.
Let $(\partial \BB_D, d_w^r)$ be  the associated ultrametric Cantor set.
Then there is an even spectral triple $(C_{\text{Lip}}(\partial \BB_D), \mathcal{H'}, \pi_\tau', \slashed{D}', \Gamma')$.
\end{prop}

The $\zeta$-function associated to the spectral triple of Theorem~\ref{prop:spectral} is given by
\[
\zeta^r_D (s) =\frac{1}{2}\Tr(|{\slashed{D}'}|^{-s})=\sum_{\lambda \in F(\BB_D)_\circ} \big(w_D^r(\lambda)\big)^s.
\]
We now want to compute the abscissa of convergence $s_r$ of the above $\zeta$-function. 

\begin{prop}\label{prop-abscissa-of-conv-O-2}
 The abscissa of convergence $s_r$ of the $\zeta$-function $\zeta^r_D  (s) $ associated to the spectral triple in Proposition \ref{prop:spectral-uneven} is~$1$.
\end{prop}

\begin{proof}
The formula for the $\zeta$-function can be written as follows:
\begin{equation}\label{eq:zeta-r}
\zeta^r_D(s)  = \frac{1}{2}\Tr(|{\slashed{D}'}|^{-s}) = \sum_{n =-1}^\infty \sum_{\lambda \in \Lambda^n} \Big( w_D^r(\lambda)\Big)^s,
\end{equation}
with the convention that a path of length $-1$ is the empty path  $ \circ $ with associated cylinder set $\Lambda^\infty$.
In order to enumerate how many paths of which weight there are in $F(\BB)_\circ$, we will use the following argument. Consider  the following formal polynomial in $D$ variables $X_1, \ldots, X_D$  with integer coefficients:
\[
 P (X_1, \ldots, X_D) = \Bigl( \sum_{i=1}^d X_i  \Bigr)^{n+1}.
\]
After expanding, each monomial is of the form $c \prod_i X_i^{\alpha_i}${{} where $c$ is a constant}. The constant $c$ counts how many partitions of $\{0, \ldots, n\}$ into $D$ (possibly empty) subsets there are, of cardinality respectively $\alpha_1, \ldots, \alpha_D$. The set of such partitions  for all possible choices of $\alpha_1,\ldots, \alpha_D$ is in bijection with 
$F^n\mathcal{B}_D$: given $\gamma = (e_1 \ldots e_n)$ (with $|e_i|=1$ and $s(e_{i-1}) = r(e_i),\  \forall i$), let $U_i = \{j \in \{0, \ldots, n\} \ : \ s(e_j) = v_i\}$.  
One sees that $\{U_i\}_{i=1}^{D}$ defines a partition of $\{0, \ldots, n\}$, and the map from $F^n\mathcal{B}_D$, the set of finite paths of $\mathcal{B}_D$ of length $n$,  to the set of such partitions is a bijection.  
Indeed, 
$$w_D^r (\gamma) = \prod_i r_i^{\alpha_i}.$$ 
Now, we see that the sum in Equation~\eqref{eq:zeta-r} can be rewritten as
\[
 \zeta^r_D (s) = {\sum_{n=-1}^\infty P(r_1, \ldots, r_D) ^{s(n+1)} }= \sum_{n = -1}^\infty \Bigl( \sum_{i=1}^D r_i^s \Bigr)^{n+1}.
\]
This is a geometric series, which converges if and only if $\sum_i r_i^s < 1$.
The function $s \mapsto \sum_i r_i^s$ is a decreasing function on $\R_+$ (since all the $r_i$ are less than $1$), and $\sum_i r_i = 1$.
Therefore, the abscissa of convergence is exactly $s_r = 1$.
\end{proof}

\begin{rmk}
 Note that this guarantees that the upper Minkowski dimension of $(\partial \BB_D, d_{w^r})$ is $1$, see~\cite[Theorem~2]{pearson-bellissard-ultrametric}.
\end{rmk}


\begin{thm}
\label{thm-Dixmier-trace-uneven-weights}
 The measure $\mu_D^r$ on $\partial \BB_D$ induced by the Dixmier trace is defined by $\mu_D^r ([\gamma]) = w_D^r (\gamma)$.
\end{thm}

\begin{proof}
Note first that for the case $\gamma = \circ$, the result follows immediately from the definitions of $\mu_D^r$ and $w_D^r$.
 Given a cylinder set 
 $[\gamma]\not= \Lambda^{\infty}$
 , we have
 \[
  \mu^r_D ([\gamma]) = \lim_{s \rightarrow s_r^+}  \frac{ \sum_{\eta \, : \, r(\eta) = s(\gamma)} \Big( w_D^r (\gamma \eta)\Big)^s}  {\zeta^r_D (s)}.
 \]
 One remark is in order: if $\gamma$ is a path of length $n$ and $0 < m < n$, then
 \[\begin{split}
 w_D^r(\gamma) & = w_D^r\bigl(r(e_1)\bigr) \prod_{i=1}^n w_D^r\bigl( s(e_i) \bigr) \\
    & = \biggl( w_D^r\bigl(r(e_1)\bigr) \prod_{i=1}^m w_D^r\bigl( s(e_i) \bigr) \biggr) \biggl( w_D^r\bigl(s(e_{m+1})\bigr) \prod_{i=m+2}^n w\bigl( s(e_i) \bigr) \biggr) \\
    & = \biggl( w_D^r\bigl(r(e_1)\bigr) \prod_{i=1}^m w_D^r\bigl( s(e_i) \bigr) \biggr) \biggl( w_D^r\bigl(r(e_{m+2})\bigr) \prod_{i=m+2}^n w_D^r\bigl( s(e_i) \bigr) \biggr) \\
    &{ = w_D^r (e_1e_2 \cdots  e_m) w_D^r(e_{m+2} \cdots e_n)}.
 \end{split}\]
%
 In particular, $w_D^r (\gamma \eta)$ is \textbf{not} $w_D^r(\gamma) w_D^r (\eta)$.
 Indeed, any path of the form $\gamma \eta$ with $s(\gamma) = r(\eta)$ can be written uniquely as $\gamma e\eta'$ where $e$ is the unique edge with $r(e) = s(\gamma)$ and $s(e) = r(\eta')$.
 By the computation above, $w_D^r (\gamma e \eta') = w_D^r (\gamma) w_D^r (\eta')$.
 {Moreover, since $\Lambda_D$ 
 has precisely one edge connecting any pair of vertices, every finite path $\eta'$ in $\Lambda $ gives rise to exactly one $e$ such that $s(e) = r(\eta')$ and $r(e) = s(\gamma)$.}  
 Therefore,
  \[
   \sum_{\eta \, : \,  r(\eta) = s(\gamma)} \bigl( w_D^r (\gamma \eta) \bigr)^s
      = \sum_{\eta' \in \Lambda} \bigl( w_D^r (\gamma) \bigr)^s \Big(  w_D^r (\eta') \Big)^s
      = \bigl( w_D^r (\gamma) \bigr)^s \alpha(s),
 \]
 where $ \alpha(s) = \sum_{\eta' \in \Lambda}  \Big(  w_D^r (\eta') \Big)^s$. 
Moreover, since $ \lim_{s \rightarrow 1^+} \alpha(s) = + \infty$, we have 
  \[
   \mu_D^r ([\gamma]) = \lim_{s \rightarrow 1^+} \Big( w_D^r (\gamma)\Big)^s\  \frac{\alpha(s)}{1+\alpha(s)} = w_D^r(\gamma).
  \]
    
\end{proof}

In particular, we do \emph{not} have $\mu_D^r = \mu_D = M$. 
This should not be completely surprising, however. The Perron--Frobenius measure $M = \mu_D$ is the unique measure on $\Lambda^\infty$ under the following assumptions: the measure is a probability measure, and $\mu_D[\gamma]$ only depends on $|\gamma|$ and $s(\gamma)$. The second assumption is not satisfied for $\mu_D^r$. 

Note also that the choice of weight $w_D^r$ does not define a self-similar ultrametric Cantor set in the sense of \cite[Definition~2.6]{julien-savinien-embedding}, since again, the diameter of $[\gamma]$ does not just depend on $|\gamma|$ and $s(\gamma)$ but also on the specific sequence of edges.

 \subsection{The Laplace--Beltrami Operator}
 \label{subsec-Laplace-Beltrami Operator-O-2}
 As  in Section  \ref{subsec-Laplace-Beltrami Operator-O-D}, the Dixmier trace associated to the spectral triple of Proposition \ref{prop:spectral-uneven} induces the probability measure $\mu_D^r(\tau)$ on the set of choice functions; thus, by the classical theory of Dirichlet forms we can define a Laplace--Beltrami operator $\Delta_s^r$ on $L^2(\partial \BB_D, \mu_D^r)$ as in Proposition~4.1 of \cite{julien-savinien-transversal} by
 \begin{equation}\label{eq:Delta-O-2}
 \langle f, \Delta_s^r(g)\rangle=Q_s(f,g)=\frac{1}{2}\int_E \Tr(\vert \slashed{D}\,\vert^{-s}[\slashed{D}, \pi_\tau(f)]^*\,[{\slashed{D}},\pi_\tau(g)]\, d\mu_D^r(\tau),
 \end{equation}
 where $\Dom Q_s=\Span \{ \chi_\gamma : \gamma\in F\BB_2\}$ is a closable Dirichlet form.

 As before, $\Delta_s^r$ is self-adjoint and has pure point spectrum, and we can  describe the spectrum of $\Delta_s^r$ explicitly. For our case we can additionally  compute the eigenvalues and the eigenfunctions of $\Delta_s^r$ as follows.
 
 \begin{thm}\cite[Theorem~4.3]{julien-savinien-transversal}\label{thm:eigen-O-2}
 Let $\Delta_s^r$ be the Laplace--Beltrami operator on $L^2(\partial \BB_2, \mu_D^r)$ given by \eqref{eq:Delta-O-2}.
 Then the eigenvalues $\{ \lambda_\eta^r \}$ and corresponding eigenspaces $\{E_\eta^r\}$ of $\Delta_s^r$ are  given by, for $\eta\in F\partial \BB_D$,
 \[
  \lambda_\eta^r=\sum_{k=-1}^{|\eta|-1}\frac{1}{G_s(\eta[0,k])}\Big(\mu_D^r[\eta[0,k+1]] - \mu_D^r[\eta[0,k])]\Big)-\frac{\mu_D^r[\eta]}{G_s(\eta)}, 
 \]
 \[
  E_\eta^r = \Span \bigg\{\, \frac{\chi_{[\eta e]}}{\mu_D^r[\eta e]}-\frac{\chi_{[\eta e']}}{\mu_D^r[\eta e']}\; : \; e\ne e',\; |e|=|e'|=1,\; r(e)=r(e')\, \bigg\},
 \]
 where $\eta[0,-1]= \circ$ and $\chi_{[\circ]} = \partial \BB_D$, $G_s(\eta[0,-1])=\frac{1}{2}\sum_{v\ne w\in\Lambda^0}\mu_D^r[v]\, \mu_D^r[w]$, and 
 for $\xi\in F\BB_D$,
 \[
   G_s(\xi)=\frac{1}{2} w_D^r(\xi)^{2-s}\sum_{e\ne e'\in r^{-1}(s(\xi))}\mu_D^r[\xi e]\, \mu_D^r[\xi e'].
 \]
 In addition, $0$ is an eigenvalue for the constant function $1$, and $\lambda_\circ = (G_s(\circ))^{-1}$ is an eigenvalue with eigenspace
 \[
  E_\circ^r = \Span \biggl\{  \frac{\chi_{[v]}}{\mu_D^r[v]} - \frac{\chi_{[v']}}{\mu_D^r[v']} \; : \; v \ne v', \ v, v' \in V_0   \biggr\}.
 \]
 \end{thm}
 
\begin{proof} 
Although Theorem 4.3 of \cite{julien-savinien-transversal} is stated only for the case when the weight function $w(\gamma)$ only depends on the length and the source of the path $\gamma$, as in Definition \ref{def-choice-of-weight-Cuntz-algebra}, a careful examination of the proof of that Theorem will show that the same  proof  
works verbatim in the case of the weight $w_D^r$.
\end{proof}

 \subsection{Eigenvalues and eigenfunctions  for the $\mathcal{O}_2$ case}
 \label{eigenvlaues-wavelets-and-eigenfunctions-O-2}

 We are going to explicitly compute here the eigenvalues 
 for the Laplace--Beltrami operator $\Delta_s^r$ in the  $D=2$ case.  Theorem \ref{thm:eigen-O-2} specializes in the case of $\mathcal{O}_2$ to give

 The formulas in Proposition \ref{eigenvalueO2unequal} below allow us to compute in principle the 
  eigenvalue associated to any finite path. However it seems difficult to get an explicit formula that covers all the cases as the calculations in full generality  are difficult to manage because of challenging bookkeeping.

\begin{lemma}\label{lem-G-s-O-2-term}
  With notation as above, for a finite path $\xi(p,q) \in F(\partial \BB_2)$  having $p$ vertices equal to $v_1$ and $q$ vertices equal to $v_2$  we have we have
 \[
   G_s(\xi(p,q)) = r^{4p+1-ps}(1-r)^{4q+1-qs}
 \]
 More generally, if $\xi$ is any path, one can write
 \[
  G_s (\xi) = \big( \mu_2^r[\xi] \big)^{4-s} r(1-r).
 \]
\end{lemma}

\begin{proof}
 We start with the second point. If $\xi$ is coded by its vertices, $\xi = (v_0, \ldots, v_{|\xi|})$ and $e \neq e'$ are vertices such that $r(e) = r(e') = s(\xi)$, then $w_D^r(\xi e) w_D^r(\xi e') = (w_D^r (\xi) )^2 r(1-r)$. Since $\mu_D^r[\xi] = w_D^r(\xi)$, we have
 \[
  G_s (\xi) = \frac 1 2 \bigl( \mu_D^r [\xi] \bigr)^{2-s} 2 \bigl( \mu_D^r[\xi] \bigr)^2 r(1-r)
 \]
 and the result follows. (Note that the factor $2$ appears because $(v_1, v_2)$ and $(v_2, v_1)$ are the two pairs in the index of the sum defining $G_s(\xi)$.)
 For the first point, we compute
 \[
   G_s (\xi(p,q)) = (1/2) \big[ r^p (1-r)^q \big]^{2-s} 2 \bigl( (r^p (1-r)^q r) (r^p(1-r)^q (1-r)  \bigr).
 \]
\end{proof}

 \begin{prop}\label{eigenvalueO2unequal}
 Let $\Delta_s^r$ be the Laplace--Beltrami operator on $L^2(\partial \BB_2, \mu^r_2)$ given by \eqref{eq:Delta-O-2} for the choice of weight induced by 
 \[
 w_2^r( v_1) = r,\ w_2^r( v_2) = (1-r),
 \]
 where $r \in [0,1]$ is fixed. (Note that the notation used above is slightly different from the notation  we used in Theorem~\ref{thm:eigen-O-2}). Let $\eta\in F\partial \BB_2$ of length $n$  be determined by the string of vertices $(v_0, \ldots, v_n)$; also  we write $(v_0, \ldots, v_k)$ for $\eta[0, k]$,  for any $k \leq n$. 
  Then we have
  \[
   \lambda_\eta^r = \frac{w^r_2(v_0)-1}{r(1-r)} + \sum_{k=0}^{n-1} \frac{(\mu_2^r[v_0, \ldots, v_k])^{s-3}}{r(1-r)} (w^r_2(v_{k+1}) - 1) - \frac{(\mu_2^r[\eta])^{s-3}}{r(1-r)}.
  \]
 \end{prop}

 \begin{proof}
 We will use the fact that if one codes $\eta$ by its vertices $\eta=(v_0, \ldots, v_{|\eta|})$, then $\mu_2^r[v_0, \ldots, v_k] = \mu_2^r[v_0, \ldots, v_i]\mu_2^r[v_{i+1}, \ldots, v_k]$, as was established in the proof of Theorem~\ref{thm-Dixmier-trace-uneven-weights}.
 Consequently, we can factor the term $(\mu_2^r [\eta[0,k+1]] - \mu_2^r [\eta[0, k]])$ as follows:
 \[\begin{split}
  \mu_2^r [\eta[0,k+1]] - \mu_2^r [\eta[0, k]] & = \mu_2^r[v_0, \ldots, v_{k+1}] - \mu_2^r[v_0, \ldots, v_{k}] \\
     & = \mu_2^r[v_0, \ldots, v_{k+1}] (\mu_2^r[v_{k}] - 1).
 \end{split}\]
 We therefore compute
  \[
    \lambda_\eta^r = \sum_{k=-1}^{|\eta|-1}\frac{1}{G_s(\eta[0,k])}  \Big( \mu^r_2[\eta[0,k+1]] - \mu^r_2[\eta[0,k]] \Big)
       - \frac{\mu^r_2[\eta]}{G_s(\eta)}, 
    \]
    that is (using point~2 of Lemma~\ref{lem-G-s-O-2-term})
  \begin{align*}
   \lambda_\eta^r &= \frac{ \mu_2^r[v_0] - 1}{G_s(\circ)} \\
     & + \sum_{k=0}^{n-1} \frac{1}{r(1-r) (\mu_2^r[v_0, \ldots, v_k])^{4-s}} \mu_2^r[v_0, \ldots, v_k] (w_2^r(v_{k+1}) - 1) \\
     & - \frac{\mu_2^r[\eta]}{r(1-r) (\mu_2^r [\eta])^{4-s} }.
  \end{align*}
  The result follows from algebraic simplifications. Note in particular that ${G_s(\circ)} = r (1-r) $.
 \end{proof}
  
One can also construct representations and wavelet spaces of $\mathcal{O}_2$ associated to the weighted Bratteli diagram $(\partial \BB_2, w_2^r)$; 
  see Theorem 3.8 of \cite{FGKP-survey}.  This is the analogue of Theorem 5.1 above for the uneven weight case. 
 We now compute the eigenspaces corresponding to the eigenvalues of Proposition \ref{eigenvalueO2unequal} above, and show that they coincide with the wavelet spaces described in \cite{FGKP-survey} Theorem 3.8. 
 In other words, we will show that if $\tilde{\mathcal{W}}_k$ are the orthogonal subspaces of $L^2(\partial \BB_2, w_2^r)$ described in 
Theorem 3.8 of \cite{FGKP-survey}, 
\[\tilde{\mathcal{W}}_k = \bigoplus_{\eta: |\eta| = k} E^r_\eta.\]

 What is done below is similar to the result of Theorem \ref{equalweightswaveletOD}, but we allow unequal weights in what follows.
 
 By evaluating the formulas given in 
 Theorem \ref{thm:eigen-O-2} we obtain: 
 
 \begin{prop}
 \label{eigenspaceO2unequal}
  Let $\Delta_s^r$ be the Laplace--Beltrami operator on $L^2(\partial \BB_2, \mu^r_2)$ given by \eqref{eq:Delta-O-2} for the choice of weight induced by the choice on the vertices $v_1$ and $v_2$ of the associated graph as 
  \[
  w_2^r( v_1)  = r_1, \ w_2^r( v_2) = r_2 = 1-r_1,
  \]
  where $r =r_1 \in [0,1]$ is fixed.  
  If we let $\circ$ denote the empty path, then the  eigenspace $E_\circ^r$ with   eigenvalue $\lambda_\circ^r$ 
  is given by 
  $\Lambda_2^{\infty}$, hence has dimension $1$ and  
  \[
    E_\circ^r = \Span \Bigl\{ \, \frac{\chi_{[v_1]}}{\mu^r_2([v_1])}-\frac{\chi_{[v_2]}}{\mu^r_2([v_2])}\, \Bigr\} = \Span \left\{ \frac{\chi_{[v_1]}}{r} - \frac{\chi_{[v_2]}}{1-r} \right\}.
  \]
  Given a   finite non-empty  path $\eta = v_{j_0} v_{j_1} \ldots v_{j_n}\in F\partial \BB_2$ with $n+1$  vertices, where $j_i \in \{1,2\} \ \forall \ i$,  the eigenspace $E_{\eta}^r$ with  corresponding eigenvalue $\lambda_{\eta}^r$ described in Proposition \ref{eigenvalueO2unequal} is given by
 \begin{align*}
   E_{\eta}^r & = \Span \Bigl\{ \, \frac{\chi_{[{\eta} e]}}{\mu^r_2[\eta e]}-\frac{\chi_{[\eta e']}}{\mu^r_2[\eta e']}\; : \; e\ne e',\; |e|=|e'|=1,\; r(e)=r(e') = s(\eta)\, \Bigr\}\\
   &= \Span \left\{ \frac{1}{(\prod_{i=0}^n r_{j_i}) r} \chi_{[v_{j_0} v_{j_1} \ldots v_{j_n} v_1]} - \frac{1}{(\prod_{i=0}^n r_{j_i}) (1-r)} \chi_{[v_{j_0} v_{j_1} \ldots v_{j_n} v_2]} \right\}.
 \end{align*}
 \end{prop}

We now show how the scaling functions generating ${\mathcal V}_0$ in Theorem 3.8 of \cite{FGKP-survey} fit into the eigenspace picture described above.

\begin{lemma} Let $r \in [0,1]$ be given, and let $\mu^r_2$ be the Markov probability measure on the infinite path space  $\Lambda_2^{\infty}$ corresponding to the weight assigning $r$ to the vertex $v_1$ and $1-r$ to the vertex $v_2.$  Let ${\mathcal V}_{-1}$ denote the space of constant functions on $\Lambda_2^{\infty}.$  Then the scaling space 
$\tilde{\mathcal V}_0$ described in Theorem 3.8 of \cite{FGKP-survey} as the span of $\{\chi_{[v_1]},\;\chi_{[v_2]}\},$ the characteristic functions of cylinder sets corresponding to the vertices, can be written as
$$\tilde{\mathcal V_0}\;=\;{\mathcal V}_{-1}\oplus E_{\circ}^r,$$ where $E_{\circ}^r$ is the eigenspace corresponding to the empty path.
\end{lemma}
\begin{proof}
We note that $\tilde{\mathcal V_0}$, being generated by the orthogonal functions $\chi_{[v_1]}$ and $\chi_{[v_2]}$,  has dimension $2.$ On the other hand, 
 the space ${\mathcal V}_{-1}$ of constant functions on $\Lambda_2^{\infty}$ has dimension $1$ and  
\[
 E_\circ^r = \Span \Bigl\{ \, \frac{\chi_{[v_1]}}{\mu^r_2([v_1])}-\frac{\chi_{[v_2]}}{\mu^r_2([v_2])}\, \Bigr\}
\]
also has dimension $1$ and is orthogonal to ${\mathcal V}_{-1}.$   It follows by a dimension count that 
$${\mathcal V_0}\;=\;{\mathcal V}_{-1}\oplus E_{\circ}^r,$$
as desired.
\end{proof}

\begin{prop} 
\label{zerowaveletlemma}
 Let $\mu^r_2$ be the Markov probability measure on the infinite path space  $\Lambda_2^{\infty}$ corresponding to the weight assigning $r$ to the vertex $v_1$ and $1-r$ to the vertex $v_2.$ Then for the corresponding representation of ${\mathcal O}_2$ on $L^2(\Lambda_2^{\infty},\mu^r_2)$ defined in Theorem 3.8 of \cite{FGKP-survey}, we have 
\[
\tilde{\mathcal W}_0= \Span_{\eta: |\eta|=0}\{E_\eta^r\},
\]
 where $E_\eta^r$ are the eigenspaces of the Laplace-Beltrami operator defined in Proposition \ref{eigenspaceO2unequal}.
\end{prop}

\begin{proof}
As in Theorem 3.8 of \cite{FGKP-survey} and Section \ref{wavelets-and-eigenfunctions-O-D} above, 
we have an inner product on $\mathbb C^2$
defined by 
$$\langle (x_j),(y_j)\rangle =\sum_{j=1}^2\overline{x_j}\cdot y_j\cdot r_j,$$ and a fixed vector 
$c^0 = c^{0,k}=(1,1).$ For $k=1,2,$ we find an orthonormal basis for $\{c^{0,k}\}^{\perp}$ denoted by $\{c^{1,k}\},$ where $
 c^{1,k}=(c_{\ell}^{1,k})_{\ell\in\{1,2\}}.$ 
 
But here, a straightfoward calculation shows that we can take 
\[
 c^{1,k}=\sqrt{r(1-r)} \Bigl( \frac{1}{r}, -\frac{1}{1-r} \Bigr),\;k=1,2.
\]
Therefore the wavelet $\psi_{1,k}$ of Theorem 3.8 of \cite{FGKP-survey} is given by
\[\begin{split}
  \psi_{1,k} & = \frac{\sqrt{r(1-r)}}{\sqrt{r_k}} \biggl[ \frac{\chi_{[v_kv_1]}}{r}-\frac{\chi_{[v_kv_2]}}{1-r} \biggr] \\
             & = \sqrt{r(1-r)r_k} \biggl[ \frac{\chi_{[v_kv_1]}}{\mu^r_2([v_kv_1])}-\frac{\chi_{[v_kv_2]}}{\mu^r_2([v_kv_2])} \biggr].
\end{split}\]
Recall that 
\[E_{v_k}^r = \Span \left \{ \frac{1}{ \mu_2^r([v_k v_1]) } \chi_{[v_k v_1]} - \frac{1}{\mu_2^r([v_k v_2])} \chi_{[v_k v_2]} \right\}\]
is a one-dimensional subspace of $L^2(\partial \BB_2, \mu_2^r)$.  Moreover, 
each vector $\psi_{1,k}$ is evidently a scalar multiple of the single spanning vector
 from $E_\eta^r$ for $\eta = v_k$ a path of length $0.$  Taking the span of the two vectors from $E_{v_1}^r$ and $E_{v_2}^r$ 
gives exactly the span of the $\psi_{1,k}$ for $k=1,2;$ since $\tilde{\mathcal W_0}$ is defined to be the span of the vectors $\psi_{1,k}$, the result follows.
\end{proof}
We now relate higher dimensional wavelet subspaces to the corresponding eigenspaces for the Laplacian:

\begin{lemma}
Let $\mu^r_2$ be the Markov probability measure on the infinite path space  $\Lambda_2^{\infty}$ corresponding to the weight assigning $r$ to the vertex $v_1$ and $1-r$ to the vertex $v_2.$ Then for the corresponding representation of ${\mathcal O}_2$ on $L^2(\Lambda_2^{\infty},\mu^r_2)$ defined in Theorem 3.8 of \cite{FGKP-survey}, we have 
\[
\tilde{{\mathcal W}}_k=\Span_{\eta: |\eta|=k}\{E_\eta^r\},
\]
where $E_\eta^r$ are the eigenspaces of the Laplacian defined in Proposition \ref{eigenspaceO2unequal}.
\end{lemma}

\begin{proof}
We prove the result by induction.  We have proved the result for $k=0$ directly.  We now suppose that for $k=n$ we have shown 
\[
 \tilde{\mathcal W}_n = \Span_{\eta: |\eta|=n} \bigl\{ E_\eta^r \bigr\},
\]
where, as defined in Theorem 3.8 of \cite{FGKP-survey}, 
\[
 \tilde{\mathcal W}_n\;=\; \Span \bigl\{ S_w(\psi_{1,k}):\;k=1,2,\;w \text{ is a word of length } n \bigr\},
\]
for $\psi_{1,1}$ and $\psi_{1,2}$ the wavelets of Lemma \ref{zerowaveletlemma}, and for $w=w_1w_2\cdots w_n$ a word of length $n,$ where $w_i\in\mathbb Z_2,$
$S_w=S_{w_1}S_{w_2}\cdots S_{w_n}$, where  (writing an infinite path $x$ as a sequence of vertices)
\[
 S_0 f(x) = \begin{cases}  r^{-1/2} f(u_2 u_3 \ldots )  & \text{if } x = (v_1 u_2 u_3 \ldots), \\
          0 & \text{else.}
          \end{cases}
          \]
and \[
 S_1 f(x) = \begin{cases}  (1-r)^{-1/2} f(u_2 u_3 \ldots )  & \text{if } x = (v_2 u_2 u_3  \ldots), \\
          0 & \text{else.}
          \end{cases}
\]
From this and the induction hypothesis, it follows that 
\[\begin{split}
 \tilde{\mathcal W}_{n+1} & = \Span \bigl\{ S_0(\tilde{\mathcal W_n}), S_1(\tilde{\mathcal W_n}) \bigr\}  \\
                    & = \Span_{\eta: |\eta|=n} \bigl\{ S_0(E_\eta^r), S_1(E_\eta^r) \bigr\},
\end{split}\]
where a typical element of $E_\eta^r$ looks like 
$$\frac{\chi_{[\eta e]}}{\mu^r_2([\eta e])}-\frac{\chi_{\eta e'}}{\mu^r_2([\eta e'])}\;.$$
Now if  $\eta = u_0 u_1 \cdots u_n$ is a path of length $n$ whose $n+1$ vertices are given in order by $u_0 u_1u_2\cdots u_{n},$
we compute directly that 
\[
 S_0\chi_{[\eta]} \;=\; \frac{1}{\sqrt{r}}\chi_{[v_1\eta]},
\]
and
\[
 S_1\chi_{[\eta']} \;=\; \frac{1}{\sqrt{1-r}}\chi_{[v_2\eta']}.
\]
Therefore we can write, for $\eta$ of length $n$ and $e$ and $e'$ of  
length $1$ with $e\not=e',$
\[\begin{split}
 S_0 \biggl( \frac{\chi_{[\eta e]}}{\mu^r_2([\eta e])}-\frac{\chi_{[\eta e']}}{\mu^r_2([\eta e'])} \biggr)
    & = \frac{1}{\sqrt{r}} \biggl[ \frac{\chi_{[v_1\eta e]}}{\mu^r_2([\eta e])}-\frac{\chi_{[v_1\eta e']}}{\mu^r_2([\eta e'])} \biggr]  \\
    & = \sqrt{r} \biggl[ \frac{\chi_{[v_1\eta e]}}{\mu^r_2([v_1\eta e])}-\frac{\chi_{[v_1\eta e'] }}{\mu^r_2([v_1\eta e'])} \biggr]
\end{split}\]
which is a constant multiple of
$$\frac{\chi_{[v_1\eta e]}}{\mu^r_2([v_1\eta e])}-\frac{\chi_{[v_1\eta e']}}{\mu^r_2([v_1\eta e'])}\; 
$$ 
which is a spanning function for the one-dimensional subspace $E_{v_1\eta}^r$. 
Similarly, $S_1(\frac{\chi_{[\eta e]}}{\mu^r_2([\eta e])}-\frac{\chi_{[\eta e']}}{\mu^r_2([\eta e'])})$ is a constant multiple of
$$\frac{\chi_{[v_2\eta e]}}{\mu^r_2([v_2\eta e'])}-\frac{\chi_{[v_2\eta e']}}{\mu^r_2([v_2\eta' e'])}, $$ 
which spans $E_{v_2 \eta}^r$.  Since all paths of length $n+1$ are of the form $v_i \eta$ for some path $\eta$ of length $n$ and some vertex $v_i$, with $i = 1,2$,
it then follows that
$$\Span_{\eta: |\eta|=n}\{S_0(E_\eta^r), S_1(E_\eta^r)\}=\;\Span_{\eta': |\eta'|=n+1}(E_{\eta'}^r).$$
But this shows that 
$$\tilde{\mathcal W}_{n+1}\;=\Span_{\eta': |\eta'|=n+1}(E_{\eta'}^r),$$ and the induction step of the proof is complete.

\end{proof}

The above results have established the following:

 \begin{thm}
Let $\mu^r_2$ be the Markov probability measure on the infinite path space  $\Lambda_2^{\infty}$ corresponding to the weight assigning $r$ to the vertex $v_1$ and $1-r$ to the vertex $v_2.$ Then for the corresponding representation of ${\mathcal O}_2$ on $L^2(\Lambda_2^{\infty},\mu^r_2)$ defined in Theorem 3.8 of \cite{FGKP-survey}, we have that the $k^{th}$-order wavelets defined there are  all constant multiples of functions of the form 
\[
\frac{\chi_{[\eta e]}}{\mu^r_2([\eta e])}-\frac{\chi_{[\eta e']}}{\mu^r_2([\eta e'])},\;|\eta|=k,\;|e|=|e'|=1,\; r(e)=r(e')=s(\eta).
\]
\end{thm}

As in the case of Lemma \ref{zerowaveletlemma}, the constant coefficient needed to transform the wavelet function $S_\eta \psi_{1,k}$ into the spanning function of $E_{\eta v_k}$
can be be computed to be 
$$\sqrt{r(1-r)}[\sqrt{r}]^j[\sqrt{1-r}]^{k+1-j},$$
where $j$ is the number of $v_1$'s appearing as vertices in the path $\eta v_k.$

\end{document}